\documentclass[sn-mathphys,Numbered]{sn-jnl}

\usepackage{graphicx}%
\usepackage{multirow}%
\usepackage{amsmath,amssymb,amsfonts}%
\usepackage{amsthm}%
\usepackage{mathrsfs}%
\usepackage[title]{appendix}%
\usepackage{xcolor}%
\usepackage{textcomp}%
\usepackage{manyfoot}%
\usepackage{booktabs}%
\usepackage[ruled,vlined,linesnumbered]{algorithm2e}%
\usepackage{algorithmicx}%
\usepackage{algpseudocode}%
\usepackage{listings}%
\DeclareMathOperator*{\argmin}{arg\,min}%
\usepackage{multirow}%
\usepackage{multicol}%
\usepackage{geometry}%
\usepackage{eucal}%
\usepackage{tikz}
\usetikzlibrary{decorations.pathreplacing}
\usetikzlibrary{fadings}
\usetikzlibrary{arrows,shapes,positioning,shadows,trees}
\usetikzlibrary{snakes}
\usepackage{rotating}
\usepackage{bbm}
\usepackage{epstopdf}

\tikzset{
  basic/.style  = {draw, text width=9cm, drop shadow, font=\sffamily, rectangle},
  root/.style   = {basic, rounded corners, align=center,
                   fill=blue!20},
  level 2/.style = {basic, rounded corners=6pt, thin,align=center, fill=green!10,
                   text width=9em},
  level 3/.style = {basic, thin, align=left, fill=pink!60, text width=5em}
}

\theoremstyle{thmstyleone}%
\newtheorem{theorem}{Theorem}[section]
\newtheorem{proposition}{Proposition}[section]%
\newtheorem{corollary}{Corollary}[theorem]

\theoremstyle{thmstyletwo}%
\newtheorem{remark}{Remark}%

\theoremstyle{thmstylethree}%
\newtheorem{definition}{Definition}%

\newcommand{\R}{\mathbb{R}}
\newcommand{\U}{\mathcal{U}}
\newcommand{\weak}{\rightharpoonup}
\numberwithin{equation}{section}

\begin{document}

\title[Shortest-path recovery from signature]{Shortest-path recovery from signature with an optimal control approach}


\author*[1,2]{\fnm{Marco} \sur{Rauscher}}\email{marco.rauscher@tum.de}

\author[1,3]{\fnm{Alessandro} \sur{Scagliotti}}\email{scag@ma.tum.de}

\author[1]{\fnm{Felipe} \sur{Pagginelli Patricio}}\email{felipe.pagginelli@tum.de}

\affil[1]{\orgdiv{CIT School}, \orgname{TU Munich}, \orgaddress{\street{Boltzmannstr. 3}, \city{Garching}, \postcode{85748}, \country{Germany}}}
\affil[2]{\orgname{Munich Data Science Institute (MDSI) 
}, \orgaddress{\city{Munich}, \country{Germany}}}
\affil[3]{\orgname{Munich Center for Machine Learning (MCML)}, \orgaddress{\city{Munich}, \country{Germany}}}


\abstract{In this paper, we consider the signature-to-path reconstruction problem from the control theoretic perspective. 
Namely, we design an optimal control problem whose solution leads to the minimal-length path that generates a given signature. 
In order to do that, we minimize a cost functional consisting of two competing terms, i.e., a weighted final-time cost combined with the $L^2$-norm squared of the controls. 
Moreover, we can show that, by taking the limit to infinity of the parameter that tunes the final-time cost, the problem $\Gamma$-converges to the problem of finding a sub-Riemannian geodesic connecting two signatures.
Finally, we provide an alternative reformulation of the latter problem, which is particularly suitable for the numerical implementation.}




\maketitle

\section{Introduction}\label{sec1}

In this paper, we suggest a technique for recovering the path with the shortest length among all paths with a given \emph{signature}, whose definition was first proposed in \cite{chen1957integration}. 
We restrict ourselves to the subset of continuous paths that have bounded variation
\begin{equation*}
    BV([0,1],\R^d) := \left\{x\in C([0,1],\mathbb{R}^d):  x(0)=0, \sup_{\mathcal{P}\in D([0,1])} \sum_{i\in \mathcal{P}} |x_{t_{i+1}}- x_{t_i}|<\infty \right\},
    \end{equation*}
where $D([0,1])$ denotes all possible finite partitions of the interval $[0,1]$. Then, the classical signature is a map from $BV([0,1],\R^d)$ to a real tensor space defined as follows for every $0\leq s <t\leq 1$:
    \begin{align*}
    &S_N^{s,t}:BV \rightarrow G^N(\mathbb{R}^d) \subset T^N(\R^d):= \bigoplus_{k =1}^N (\mathbb{R}^d)^{\otimes k }\\
        &S_N^{s,t}: x(\cdot) \mapsto \bigg(1, \int_{s\leq u\leq t}dx_u,\dots, \int_{s\leq u_1 \leq \dots \leq u_N\leq t} dx_{u_1} \otimes \dots \otimes dx_{u_N}\bigg).
    \end{align*}
The image of $BV$ under the signature map forms a subset in the tensor space, which we denote by $G^N(\mathbb{R}^d)$. We refer the reader to \cite[Chapter~7]{friz2010} for a gentle introduction to the rich algebraic structure of this space.
One of the properties described in this reference is that the image of the signature map forms a Lie group, and therefore it can be seen as a topological manifold. 

The reason for the interest in signatures comes from many different applications. One recent problem in finance is to generate time series in order to find hedging strategies for different scenarios, as can be seen in \cite{buehler2020data, buehler2019deep}. This method needs several different paths that come from the same distribution, representing different situations to be taken into account in the hedging strategy. 
One approach for this task is to generate several log-signatures sampled from the same distribution, and then try to invert them back into paths. For further recent applications to mathematical finance, we refer the interested reader to \cite{futter2023signature,horvath2023optimal}. 
Generally, the concept of signature comes from rough path theory (see \cite{lyons1998differential}), and it is now exploited in the Machine Learning community in different ways. An interesting direction is the embedding in kernel theory, where the signature serves as a feature map. The corresponding Reproducing Kernel Hilbert Space is universal to continuous functions on the path space, while the mean embedding characterizes distributions over the path space \cite{chevyrev2022signature,salvi2021higher}. An approximation scheme to derive the signature itself was for example given in \cite{kiraly2019kernels}. 
Moreover, we mention that in \cite{kidger2020neural} a continuous-time analogue of Recursive Neural Networks (RNN) was introduced to deal with predictions based on observations from rough paths. 
Therefore, in \cite{kidger2020neural} the authors considered controlled integral equations of the form
\begin{equation*}
    z(t) = z(t_0) + \int_{t_0}^t f_{\theta}(z(s))\, dx_s,
\end{equation*}
where $s\mapsto x_s\in \R^v$ is a spline constructed using instant evaluations of a rough path. In their model, $f_{\theta}:\R^w\to \R^{w\times v}$ is a parametric vector field, and $\theta$, \emph{which is not allowed to depend on the time}, is chosen by solving the minimization problem related to the design of the dynamics. 
We insist on the fact that in our problem, the controlled fields are fixed and prescribed by the geometry of the space of signatures, and we obtain the path (or, more precisely, its velocity) by solving an optimal control problem.


The aim of the present article is to find the shortest connection between two signatures contained in the manifold, i.e., we are interested in a geodesic problem in the space of signatures. 
For this reason, we equip the space $G^N(\R^d)$ with a sub-Riemannian structure induced by the control system:
    \begin{align}\label{eq:int_subR_sys}
        \dot{\xi}(s) &= \sum_{i=1}^d a_i(s) U_i(\xi(s)) \ \mbox{for a.e. } s\in[0,1],\ a_i \in L^2([0,1],\R),  
    \end{align}
where $U_1,\ldots, U_d$ are the vector fields providing at every point an \emph{orthonormal frame} for the sub-Riemannian distribution, and they are defined as follows:
    \begin{equation*}
    U_i:G^N(\mathbb{R}^d) \to T^N(\mathbb{R}^d), \  U_i:\xi \mapsto \xi \otimes e_i,\ i\in\{1,\dots,d\}.
    \end{equation*}
For the basic notions on sub-Riemannian geometry, we recommend \cite[Chapter~3]{agrachev2019comprehensive}. However, we report that in the present paper, we do not make use of advanced concepts of sub-Riemannian geometry. \\
Given an admissible control $a=(a_1,\ldots,a_d)\in L^2([0,1],\R^d):=  \U$, we can consider the solution of \eqref{eq:int_subR_sys} with starting point $\xi(0)=(1,0,\ldots,0)$.
Then, if we interpret $a_i, i\in \{1,\dots,d\}$ as the $i-$th component of the time derivative of a path $x\in {H}^1([0,1],\R^d)$, i.e.,
\begin{equation} \label{eq:cor_path_intro}
    \begin{cases}
        \dot x_i(s) = a_i(s) \quad \mbox{a.e. in }  [0,1] & i=1,\ldots,d  ,\\
        x(0) = 0,
    \end{cases}
\end{equation}
it turns out that $S_N^{0,1}(x) = \xi(1)$ {(see \cite[Proposition~7.8]{friz2010})}. 
In Section~\ref{sec:prereq} we explain that every signature in $G^N(\R^d)$ can be generated by a path belonging to the space $H^1([0,1],\mathbb{R}^d)$.
While Rashevsky–Chow Theorem ensures the existence of a connection between any two signatures in the manifold, it does not provide the shortest one. Since there are several possible curves linking $g_0$ and $g_1$ in $G^N(\mathbb{R}^d)$, the aim of the present contribution is to provide a numerical approach to approximate the length-minimizer geodesic.

Generally, the concept of geodesics is related to the definition of distance between two points of a manifold, and to the notion of the length of a curve. A length-minimizer geodesic is a curve connecting two elements in the manifold whose length is minimal. 
Roughly speaking, the idea of constructing such geodesics is to move efficiently in the manifold along the directions specified by the orthonormal vector fields.
In our case, the orthonormal fields are the $U_1,\ldots,U_d$ appearing at the right-hand side of \eqref{eq:int_subR_sys}, and we observe that the weights (or \emph{controls}) $a_i,\ i\in \{1,\dots,d\}$ determine the movement in $G^N(\mathbb{R}^d)$. 
Given a control $a\in \U$, the length of the corresponding curve obtained by solving \eqref{eq:int_subR_sys} with $\xi(0)=g_0$ is  
\begin{equation} \label{eq:length_intro}
    \ell(\xi) = \int_0^1 |a(s)|_2\,ds=
    \int_0^1 \sqrt{ \sum_{i=1}^d |a_i(s)|^2 } \,ds.
\end{equation}
This length leads to the definition of the Carnot-Caratheodory metric, that induces a topology on $G^N(\mathbb{R}^d)$ equivalent to the standard one related to the Lie group structure, and it turns $G^N(\mathbb{R}^d)$ into a geodesic space (see \cite[Chapter~7]{friz2010}).
As a consequence of this notion of length, we relate the length-minimizing curves $s\mapsto \xi(s) \in G^N(\mathbb{R}^d)$ connecting $g_0=(1,0,\ldots,0)$ and $g_1=\bar g$ to the corresponding optimal control $a\in \U$, and, using \eqref{eq:cor_path_intro}, we construct a path in $s\mapsto x(s)\in \R^d$ such that $S_N^{0,s}(x(\cdot))=\xi(s)$ for every $s\in [0,1]$. In particular, we obtain that the above-mentioned curve $x\in H^1([0,1],\R^d)$ is \emph{the shortest path satisfying $S_N^{0,1}(x(\cdot)) = \bar g$}.
Having this in mind, we first introduce an appropriate optimal control problem to find the length-minimizer geodesics in $G^N(\mathbb{R}^d)$, and we solve it numerically. In a second step, we recover the corresponding path in $H^1$ using \eqref{eq:cor_path_intro} and the optimal controls computed in the first phase. 

Reconstructing the path from a given signature was already considered in the literature in e.g. \cite{pfeffer2019learning}, where the authors tackle the path recovery problem from the algebraic geometry perspective. The authors focus on signatures up to order three and propose a method for the exact recovery of paths with a predetermined number of pieces, that constitute a dictionary. These pieces can have different structures, e.g., polynomial or linear. 
Moreover, the paper suggests using an optimization routine for the recovery problem when the dictionaries of the underlying paths are large. 
In this scenario, the identifiability of a unique path given a predetermined signature is lost, hence in \cite[Section~9]{pfeffer2019learning} the extra ``degrees of freedom'' are saturated by penalizing the length.
Techniques from algebraic geometry have been fruitfully applied to signatures also by \cite{amendola2019varieties}, where the authors already proved a formula for the exact reconstruction of a path, but describe limitations of their approach in terms of the maximal order of the signature as well as the maximal number of segments of a path. 
The problem of path recovery was also addressed in \cite{lyons2018inverting} and \cite{lyons2017hyperbolic}, where the authors deal with the recovery of a piecewise linear path from a given signature, and they focus on the recovery of the direction as well as the length of each piece from the signature. 
In \cite{lyons2017hyperbolic} the differential equation fulfilled by the signature is solved for paths that are projected to a hyperbolic space. The benefit of this strategy for the reconstruction is that the curvature of the space provides information about the length and direction of a piece of a piecewise linear function. 
In \cite{lyons2018inverting}, the authors use a parametric form for each segment of a piecewise linear function, and adjust the parameters to match a given signature. 
{Another recent direction dealing with the inversion of the signature is given by the insertion method introduced in \cite{chang2019insertion} and investigated by e.g. \cite{fermanian2023insertion}.}

The tool we employ to recover the length-minimizing paths from a given signature is the algorithm from \cite{sakawa1980}, originally introduced to numerically address optimal control problems with an iterative procedure. In more detail, this algorithm relies on the necessary optimality conditions expressed in form of Pontryagin's Maximum Principle (described e.g. in \cite[Chapter~12]{agrachev2004control}). 
This paper merges tools from optimal control and the signature literature, and we focus on the applicability of several results from the classical optimal control theory to signatures. Our approach is advantageous since we do not have theoretical constraints concerning the order of the signature or the dimension of the path. 
Nevertheless, the implementations given in the numerics section contain natural limitations due to the fact that we deal with a discretized problem.
Another advantage of our method is that we do not focus on the piecewise linear structure of a discretized path, therefore, we can theoretically deal with more abstract paths.

First, we review several essential results and properties of the signature. Then, we introduce the related optimal control problem, and we propose several formulations, rigorously explaining how they relate to each other.
Moreover, we employ techniques of $\Gamma$-convergence to relax the hard-constraint of matching exactly the target signature. Namely, we introduce a weighted non-negative term that penalises the mismatch between the produced signature and the target one (soft-constraint), and we prove that we recover the original hard-constrained problem if we let the weight tend to infinity. 
Subsequently, we state the Pontryagin Maximum Principle (PMP), which gives us the necessary minimality conditions for the optimal control problem. The remaining sections deal with versions of the algorithm from \cite{sakawa1980} which exploits the PMP to find the shortest connection between two signatures. 
More precisely, we investigate the recovery of the geodesic in $G^N(\R^d)$ and of the related path.
The last section gives some illustrative examples of realizations of stochastic processes and concludes with a comparison to their shortest path that we were able to recover from the signatures.

\section{Preliminaries on Signatures}\label{sec:prereq}

In this section, we introduce some preliminary results about signatures that are helpful to embed this theory in the optimal control framework. 
Moreover, this section should motivate our approach for solving the optimization problem that we tackle in the subsequent sections. Although we follow the structure of Chapter 7 in \cite{friz2010}, we refer to this book for a more in-depth introduction to the topic. 

Since the signature maps paths to an abstract tensor space, we need to clarify what paths we consider in this work before properly defining the signature. 
As already mentioned in the introduction, when we optimize for the geodesics, we restrict ourselves to the subset of continuous paths whose first derivative is in $L^2([0,1],\mathbb{R}^d)$. More precisely, we consider the space
    \begin{equation*}H^1([0,1],\mathbb{R}^d) := \{x\in C([0,1],\mathbb{R}^d): \ x(0)=0,\ \dot{x}\in L^2([0,1],\ \mathbb{R}^d)\}.\end{equation*}
On the other hand, we the signature can be defined for more general continuous paths that have bounded variation and that are denoted by
    \begin{equation*}BV := \{x\in C([0,1],\mathbb{R}^d): \ x(0)=0, \sup_{D([0,1])} \sum_{i\in D([0,1])} |x_{t_{i+1}}- x_{t_i}|<\infty \}.\end{equation*}
We are now prepared to define the signature according to \cite[Definition~7.2]{friz2010}.  

\begin{definition}[Signature]

For a given d-dimensional path $x\in BV([s,t],\mathbb{R}^d)$ the signature up to order $N$ is given by
\begin{equation*}S_N(x)_{s,t} := \bigg(1, \int_{s\leq u\leq t}dx_u,\dots, \int_{s\leq u_1 \leq \dots \leq u_N\leq t} dx_{u_1} \otimes \dots \otimes dx_{u_N}\bigg)\in \bigoplus_{k=1}^N (\mathbb{R}^d)^{\otimes k }\end{equation*}
\end{definition}

We notice that the state space of the signatures up to order $N$ is contained in the tensor space $T^N(\mathbb{R}^d) := \bigoplus_{k=1}^N (\mathbb{R}^d)^{\otimes k }$. We denote the exact state space of signatures generated from $BV$ as $\tilde G^N(\mathbb{R}^d)$. The product in this space is called the tensor product between two elements $g,h\in T^N(\mathbb{R}^d)$ and is defined as follows:
    \begin{equation*}
    \pi_k(g \otimes h) = \sum_{i=1}^k \pi_i(g) \otimes \pi_{k-i}(h),\ \forall k\in [N],
    \end{equation*}
where, for arbitrary $g,h\in T^N(\mathbb{R}^d)$, we indicate the projection to the $k$-th component of $g\in T^N(\mathbb{R}^d)$ as $\pi_k(g) = g^k$. 

One of the crucial properties is that curves of signatures solve an integral equation for all paths contained in $BV$.
Moreover, in the case the underlying path has the Sobolev regularity $H^1$, the corresponding curve of signatures solves the ODE that we report below. We observe that considering Sobolev-regular paths is not restrictive for our scope, as we shall see in the sequel of this section. 
We denote the corresponding space of signatures generated from $H^1$ as ${G}^N(\mathbb{R}^d)$.   

\begin{theorem}\label{control_sys}
For $x\in H^1([0,1],\mathbb{R}^d)$ such that $a_i := [\dot{x}]_i$, the curve of signatures, $\xi(\cdot):=S_N(x)_{0,\cdot}$, solves the following ODE
    \begin{equation} \label{eq:control_syst}
        \begin{split}
        \dot{\xi}(s) &= \sum_{i=1}^d a_i(s) U_i(\xi(s)) , s\in[0,1], \\
        \xi(0) &= (1,0,\dots,0)\in G^N(\mathbb{R}^d),
        \end{split}
    \end{equation}
where the $d$ vector fields $U_1,\ldots,U_d$ are defined as 
    \begin{equation}\label{vec_fields}
        U_i:{G}^N(\mathbb{R}^d) \rightarrow T^N(\mathbb{R}^d), \quad  U_i:\xi \mapsto \xi \otimes e_i,\ i\in\{1,\dots,d\}.
    \end{equation}
\end{theorem}

\begin{proof}
    See \cite[Proposition~7.8]{friz2010}.
\end{proof}

The vector fields defined in \eqref{vec_fields} have sublinear growth, as we can observe the following estimate
    \begin{equation}\label{sublingr}
        \sup_{i\in\{1,\dots,d\}} |U_i(\xi)| = \sup_{i\in\{1,\dots,d\}} |\xi \otimes e_i| \leq |\xi|+1,
    \end{equation}
for every $\xi$.
{Furthermore, given controls $(a_1, \ldots, a_d)\in L^2([0,1], \R^d)$, the condition \eqref{sublingr} implies that there exists a $C>0$ such that }
    \begin{equation}\label{assumpt_bounded_sig}
        ||\xi||_{C^0}\leq C, \ \text{where } ||\xi||_{C^0} := \sup_{s\in[0,1]} |\xi(s)|.
    \end{equation}
Combining \eqref{assumpt_bounded_sig} and \eqref{sublingr} we can derive the following bound
    \begin{equation}\label{bound_g}
        \|\dot{\xi}\|_{L^2}^2 \leq \int_0^1 \sum_{j=1}^d |U_i(\xi(s))|^2 |a_j(s)|^2 \, ds \leq (1+C)^2  C_a^2 d,
    \end{equation}
if $||a||_{L^2}\leq C_a$. Note, that \eqref{assumpt_bounded_sig} and \eqref{bound_g} show that $g$ takes elements in the Sobolev space $\mathbb{X}:= H^1([0,1],G^N(\mathbb{R}^d))$. We denote this space by $\mathbb{X}$ to avoid confusion with $H^1([0,1],\mathbb{R}^d)$, which we sometimes abbreviate as $H^1$ .

Next, we recall Chen's relation which relates a concatenation of two different paths on the level of signatures. 

\begin{theorem}[Chen]\label{Chen}

Given $x\in BV([0,T_1],\mathbb{R}^d)$ and $y \in BV([0,T_2],\mathbb{R}^d)$, define the concatenation $x\sqcup y \in BV([0,T_1+T_2],\mathbb{R}^d)$ of these two paths as follows. 

\begin{equation*}x \sqcup y(s) = \begin{cases}  x(s), \ s\in[0,T_1] \\
y(s-T_1) - y(0) + x(T_1) , \ s\in [T_1,T_1+T_2].
\end{cases}\end{equation*}

Then the following equation holds:

\begin{equation*}
    S(x \sqcup y) = S(x)\otimes S(y).
\end{equation*}
\end{theorem}

\begin{proof}
     See \cite[Theorem~7.11]{friz2010}.
\end{proof}

\subsection{Sub-Riemanian Structure} 

In this subsection, we explain the inherent structure of the vector fields which can be identified with a specific Lie algebra. This identification provides the natural structure of more general vector fields and allows us to use additional types of vector fields as introduced in \ref{control_sys}. This can lead to a more precise description of the control system. 

As described in \cite[Chapter~7]{friz2010} the state space ${G}^N(\mathbb{R}^d)$ forms a Lie group w.r.t. $\otimes$ and therefore it is a smooth manifold. This means that there exists a corresponding Lie algebra $\mathfrak{t}(\mathbb{R}^d)$ which can be seen as the tangent space in the neutral element of the Lie group. Besides the scalar multiplication and the standard plus operations, the Lie algebra is additionally closed under the bracket operation defined as follows

\begin{definition}[Bracket]\label{brackets}

Consider two elements $g,h\in \mathfrak{t}(\mathbb{R}^d)$, the bracket is then defined as
\begin{equation*}
    (g,h) \rightarrow [g,h]:= g\otimes h - h\otimes g \in  \mathfrak{t}(\mathbb{R}^d).
\end{equation*}
\end{definition}

With this definition, we can also see the following two implications for all $g,h,k\in \mathfrak{t}(\mathbb{R}^d)$:
\begin{equation*}
    [g,h] = -[h,g] \text{ and } [g,[h,k]]+[h,[k,g]] + [k,[g,h]] = 0.
\end{equation*}
Any vector space $(\mathcal{V},+,.)$ equipped with a bracket as defined above is called a Lie algebra. Thus, $(\mathfrak{t}^N(\mathbb{R}^d),+,.,[\cdot,\cdot])$ is a Lie algebra. Next, we extend the notion of a single bracket to formulate the definition of nilpotent Lie algebra. 

\begin{definition}[Nilpotent Lie algebra of order $N$]\label{lie_algebra}
    Consider $\mathfrak{g}^N(\mathbb{R}^d)\subset \mathfrak{t}^N(\mathbb{R}^d)$ defined as the smallest Lie subalgebra of $\mathfrak{t}^N(\mathbb{R}^d)$ which contains $\mathbb{R}^d$:
    \begin{equation*}
    \mathfrak{g}^N(\mathbb{R}^d) = \mathbb{R}^d \oplus [\mathbb{R}^d,\mathbb{R}^d]\oplus \dots \oplus [\mathbb{R}^d,[\dots,[\mathbb{R}^d,\mathbb{R}^d]]].
    \end{equation*}
    The set $\mathfrak{g}^N(\mathbb{R}^d)$ is called the nilpotent Lie algebra of order $N$, whose basis elements can be formed as 
    \begin{align*}
        e_\alpha &= [e_{\alpha_1},[e_{\alpha_2},[\dots,[e_{\alpha_{k-1}},e_{\alpha_k}]]]
        \in [\mathbb{R}^d,[\dots,[\mathbb{R}^d,\mathbb{R}^d]]] \subset (\mathbb{R}^d)^{\otimes k}, k\leq N.
    \end{align*}
\end{definition}

One way of establishing a connection between the Lie group and the Lie algebra defined above is via the exponential map. 

\begin{definition}[Exponential]\label{exponential}
The exponential map is defined by 
\begin{align*}
    \exp:\, & \mathfrak{t}^N(\mathbb{R}^d)\rightarrow 1+\mathfrak{t}^N(\mathbb{R}^d),\\
     \exp:\, &a \mapsto 1+\sum_{i=1}^N \frac{a^{\otimes k}}{k!}.
\end{align*}
This sum is only taken up to the $N$-th order since we only consider tensors up to order $N$. 
\end{definition}

A simple calculation gives us that the signature of a linear path coincides exactly with the exponential function of Definition \ref{exponential}, i.e.,
\begin{equation}\label{lin_example}
    x:[0,1] \rightarrow \mathbb{R}^d,\ x:t \rightarrow tv \implies S(x) = \exp(v).
\end{equation}
 This relation is exploited by the Rashevsky–Chow Theorem which determines the precise state space of the signatures $G^N(\mathbb{R}^d)$. 
\begin{theorem}\label{Chow}
The following statements hold:
\begin{enumerate}
    \item \emph{(Rashevsky-Chow)} For any $g\in \exp(\mathfrak{g}^N(\mathbb{R}^d))$ there exist $v_1,\dots,v_m \in \mathbb{R}^d$ such that
\begin{equation*}g=e^{v_1}\otimes \dots \otimes e^{v_m}.
\end{equation*}
\item If we set $G^N(\mathbb{R}^d):= \{ S_N(x)_{0,1}:x\in H^1([0,1],\mathbb{R}^d)\}$, we have that 
        \begin{equation*}G^N(\mathbb{R}^d)=\exp(\mathfrak{g}^N(\mathbb{R}^d)),\end{equation*}
    where $\mathfrak{g}^N(\mathbb{R}^d)$ is the closed sub-Lie algebra of $\mathfrak{t}^N((\R^d),\otimes)$ introduced in Definition~\ref{lie_algebra}. 
\end{enumerate}
\end{theorem}

\begin{proof}
     For the first result, see \cite[Theorem~7.28]{friz2010}.
     Regarding the second point, in \cite[Theorem~7.30]{friz2010} the statement was originally proved for the space of paths $x\in BV([0,1],\mathbb{R}^d)$, i.e. continuous functions with finite variation. More precisely, if we consider the set of all signatures given as \begin{equation*}
     \Tilde{G}^N(\mathbb{R}^d):= \{S_N(x)_{0,1}:x\in BV([0,1],\mathbb{R}^d)\},
     \end{equation*}
   then the argument in \cite[Theorem~7.30]{friz2010} shows the following relations:
        \begin{equation*}
        \exp(\mathfrak{g}^N(\mathbb{R}^d)) \subset \Tilde{G}^N(\mathbb{R}^d) \subset  \exp(\mathfrak{g}^N(\mathbb{R}^d)). \end{equation*}
   In our case, we have the inclusions $\exp(\mathfrak{g}^N(\mathbb{R}^d)) \subset G^N(\mathbb{R}^d) \subset \Tilde{G}^N(\mathbb{R}^d)$, from where we can conclude. The first inclusion descends from Rashevsky–Chow, since every element in $\exp(\mathfrak{g}^N(\mathbb{R}^d))$ can be constructed as the signature of the concatenations of linear paths, which are contained in $H^1$.
   The second inclusion follows from the fact that
    $H^1 \subset BV 
    $.
\end{proof}


The last result shows that any signature generated by a path in $BV$ can be represented as the signature of a piecewise linear path or of a path contained in $H^1$. 
Moreover, there is a deep link between the left-invariant vector fields and the nilpotent Lie algebra. 

\begin{remark}[Sub-Riemannian structure of $G^N(\mathbb{R}^d)$]\label{sub_riemann}
An additional consequence of the Rashevsky–Chow Theorem is that any two points in $G^N(\mathbb{R}^d)$ can be connected by a curve 
\begin{equation}\label{curve}
    s\mapsto \xi(s) \in H^1([0,1],G^N(\mathbb{R}^d))
\end{equation} 
such that for every $s\in[0,1]$ the corresponding element $\xi(s)$ on this curve stays in $G^N(\mathbb{R}^d)$. 
This is because the vector fields $\xi\mapsto U_i(\xi) = \xi \otimes e_i$ $i=1,\ldots,d$ generate a Lie algebra that, point by point, coincides with the tangent space: 
\begin{equation}\label{tangent_lie}
    \mathcal{T}_{g} G^N(\mathbb{R}^d) = \mathrm{Lie}[U_1,\dots,U_d]|_g.
\end{equation}
A sub-Riemannian metric on $\mathcal{T}_g G^N(\mathbb{R}^d)$ can be introduced by taking $\{U_1,\dots,U_d\}$ to be, point by point, an orthonormal frame. 
\end{remark}

\subsection{Carnot-Caratheodory Metric}

We now recall the result that introduces the Carnot-Caratheodory (CC) metric. 

\begin{theorem}\label{CC}
    The Carnot-Caratheodory norm for any element $g\in G^N(\mathbb{R}^d)$ is given by 
    \begin{equation} \label{eq:inf_cc}
    \|g \|_{CC}:=\inf \bigg\{ \int_0^1|\dot{y}(s)|ds: y\in H^1([0,1],\mathbb{R}^d) \text{ and } S_N(y)_{0,1}
     = g\bigg\}.
     \end{equation}
    This infimum is always achieved, and the corresponding path can be parametrized to be Lipschitz continuous and of constant speed $|\dot{y}^*(r)|=c$ for every $ r\in[0,1]$.\\
    Moreover, $G^N(\mathbb{R}^d)$ equipped with the CC metric is a geodesic space. More precisely,
    given two elements $g,h\in G^N(\mathbb{R}^d)$, a geodesic connecting them is
    \begin{equation*}t\in [0,1] \rightarrow \mathcal{Y}_t:= g\otimes S_N(y^*)_{0,t}\end{equation*}
where $y^*$ is the path realizing the $\inf$ in \eqref{eq:inf_cc}  associated to $\|g^{-1}\otimes h\|_{CC}$.
\end{theorem}

\begin{proof}
    See \cite[Theorem~7.32]{friz2010}. We mention that the fact that the infimum is taken over a non-empty descends from Rashevsky–Chow, since it implies that there is always a piecewise linear path whose signature coincides with $g$. 
    The second part of the statement descends from \cite[Proposition~7.42]{friz2010}.
\end{proof}

In the final Remark of this section below we mention the starting point for our approach to determine the shortest-length paths with a given signature. In other words, this means finding the minimizer of the CC metric of a given signature.

\begin{remark}

The geodesic from above satisfies the differential equation of Theorem \ref{control_sys} for the left-invariant vector fields  $\xi\mapsto U_i(\xi) = \xi \otimes e_i$ $i=1,\ldots,d$ whose Lie algebra coincides with the entire tangent space at the point $\xi$. 
If we consider an element tangent to the $span\{U_1(\xi,\dots,U_d(\xi)\}$, then we can measure its norm as follows:
    \begin{equation*}
    v= \sum_{i=1}^d a_i U_i(\xi) \implies \|v\| = \sqrt{\sum_{i=1}^d a_i^2}.
    \end{equation*}
Using this norm we can measure the length of the whole curve in \eqref{curve} as the CC norm introduced in Theorem \ref{CC} suggests.
\end{remark}

\section{Embedding in Optimal Control Theory}

\subsection{Problem Formulation}\label{section_prob_form}

This section is devoted to the formulation of the optimal control problem related to the shortest path among all paths in $H^1$ whose signatures coincide. First of all, we have to define the exact objective functional whose minimizers are paths of minimal length such that their signatures coincide with an assigned one.
We formulate this task as an optimal control problem where the controls $a_i := [\dot{x}]_i \in\mathbb{R}, i \in\{1,\dots,d\}$ are the first derivative of the path $x\in H^1([0,1],\mathbb{R}^d)$.
We already described the relation between the controls and the signatures, so that we can interpret the integral over the controls as the length of the corresponding curve in the signature space $\xi \in \mathbb{X}:= H^1([0,1],G^N(\mathbb{R}^d))$. Namely, we have that:
    \begin{align}\label{def_len}
        \ell(\xi) = \int_{t=0}^1 |\dot{x}(t)| dt =  \int_{t=0}^1 \sqrt{\sum_{i=1}^d |a_i(t)|^2} dt.
    \end{align} 
It is important to observe that the length is invariant under reparametrizations. To see this, we consider a reparameterization, which is a monotone increasing, smooth function $\psi:[0,T]\rightarrow [0,1]$ and compose as follows
    \begin{align*}
         \ell(\xi) = \int_{t=0}^1 |\dot{x}(t)| dt  = \int_{\psi^{-1}(0)=0}^{\psi^{-1}(1)} |\dot{x}(\psi(s))| \ |\dot{\psi}(s)| ds = \ell(\xi \circ \psi) .
    \end{align*} 
Given this notion of the length of a path, the optimal control problem reads as 
    \begin{equation}\label{min}
        \min_{\xi\in \mathbb{X}} \ell(\xi),
    \end{equation}
such that $\xi \in \mathbb{X}$ fulfills the constraints defined in \eqref{control_sys} and additionally $\xi(1) = \bar g$. 
{As already explained in Section~\ref{sec:prereq}, there are always admissible controls due to the Rashevsky–Chow result (Theorem~\ref{Chow}). 
Moreover, we note that minimizing the length over admissible controls coincides exactly with finding the norm introduced in Theorem~\ref{CC}.}
For numerical matters, it can sometimes be more convenient to work with a similar functional strictly related to the length whose minimizers coincide with the ones in \eqref{min}. The \emph{energy functional} is often denoted as $J: L^2([0,1],\R^d) \rightarrow \mathbb{R}$ takes the following form:
    \begin{equation}\label{def_energy}
        J(a)  = \frac{1}{2} \int_0^1 |a(t)|^2 dt = \frac{1}{2} \int_0^1 |\dot{x}(t)|^2 dt.
    \end{equation}
The length $\ell$ and the energy functional $J$ can be related by the following Theorem.

\begin{theorem}\label{equivalence_of_probs}
    Consider Sobolev-regular curves $\mathbb{X} \ni \xi:[0,T]\to G^N(\R^d) $ such that $\xi(0)=(1,0,\dots,0)$ and $\xi(T)=\bar{g}$, and let $a\in \U$ be the admissible control that generates $\xi$. Then the following conditions are equivalent:
    \begin{itemize}
        \item[1)] The admissible controls $a\in \U$ are minimizing the energy $J.$ 
        \item[2)] A curve $\xi \in \mathbb{X}$ is a minimizer of $\ell$ whose controls fulfil $|a(t)| = c>0$, for almost every  $t\in [0,T]$.
        \item[3)] Consider the problem of finding the minimal time $\hat T$ such that there exists a Sobolev-regular curve $\hat \xi: [0,\hat T]\to G^N(\R^d)$ with $\hat \xi (0)=(1,0,\ldots,0)$ and  $\hat \xi (\hat{T}) =\bar{g}$, whose controls fulfils $|\hat a(t)| = 1$ for almost every $t\in [0,\hat{T}]$. Then, $\xi$ is a reparametrization of $\hat \xi$.
    \end{itemize}
\end{theorem}
\begin{proof}
    We start by showing that with a minimizer of 2) we can construct a minimizer of 3).
    Let $\xi:[0,T]\to G^N(\R^d)$ be a length-minimizer connecting $\xi(0)$ to $\xi(T)=\bar g$, whose constrol satisfies $|a(t)|=c>0$ a.e. in $[0,T]$. 
    Let us set $\hat T:= cT$, and let us define $\hat \xi:[0,\hat T]\to G^N(\R^d)$ as $\hat \xi(t):= \xi(t/c)$.
    We observe that
    \begin{equation*}
        \frac{d}{dt}\hat \xi (t) = \frac1c \dot \xi(t/c)=
        \sum_{i=1}^d \frac{a_i(t/c)}{c} U_i(\xi(t/c))
        = \sum_{i=1}^d \frac{a_i(t/c)}{c} U_i(\hat \xi(t)),
    \end{equation*}
    yielding $|\hat a(t)|= \frac1c |a(t/c)|=1$ a.e. in $[0,\hat T]$. In particular, this implies that $\ell(\xi) = \ell(\hat \xi) = \hat T$, where the first identity follows from the invariance of the length for reparametrizations. We claim that $\hat \xi$ is a curve in $G^N(\R^d)$ with unitary speed that arrives to $\bar g$ in minimal time. Indeed, if $\hat{\xi} ':[0,\hat{T}']\to G^N(\R^d)$ had unitary speed and arrived at $\bar g$ in time $\hat T'< \hat T$, then we would have $\ell(\hat \xi') = \hat T'<\hat T= \ell(\xi)$, contradicting the fact that $\xi$ is the length-minimizer.\\
    We now show that 3) implies 2).
    Let $\hat \xi:[0,\hat T] \to G^N(\R^d)$ be a curve with unitary speed that arrives at $\bar g$ in minimal time.
    We claim that it is a reparametrization of a length-minimizer curve $\xi:[0,T]\to G^N(\R^d)$ that joins $\xi(0)=(1,0,\ldots,0)$ to $\xi(T)=\bar g$ with length $\ell(\xi)=\ell(\hat \xi) = \hat T$.
    Assume that $\xi$ is not length-minimizer, i.e. there exists $\xi':[0,T]\to G^N(\R^d)$ joining the neutral element to $\bar g$ with $\ell(\xi')<\ell(\xi)$.
    Hence, as detailed above, it is possible to construct a reparametrization $\hat \xi ':[0,\hat T']\to G^N(\R^d)$ with unitary speed, such that $\hat \xi'(\hat T')=\bar g$ and $\hat T'= \ell(\xi')$. However, this is in contradiction with the time-minimality of $\hat \xi$.\\
     The equivalence of 1) and 2) is a consequence of Cauchy Schwarz inequality and is proved in \cite[Lemma~3.64]{agrachev2004control}.
\end{proof}
The previous result shows that minimizing the length is indeed equivalent to minimizing the energy. 
In this case, the energy-minimizer is no more invariant for reparametrizations (as it is the case for the length functional), and in fact the curve with minimal energy is the length-minimizer with constant speed. 
Finally, we can state the optimization problem of interest:
    \begin{equation}\label{ori_prob}
        \begin{cases}
        \min_{a\in L^2([0,1],\R^d)}\ \ J(a), \\
        \dot{\xi}(s) = \sum_{i=1}^d a_i(s) U_i(\xi(s)) ,\ s\in[0,1], \\
        \xi(0) = (1,0,\dots,0)\in G^N(\R^d), \\
        \xi(1) = \bar{g}\in G^N(\R^d).
        \end{cases}
    \end{equation}
An alternative version of this problem relaxes the ``hard'' endpoint constraint $\xi(1) = \bar{g}$ and integrates it in the cost functional, resulting in a more numerically-effective approach.
This version of the problem reads as follows: 
    \begin{equation}\label{gamma_opt_prob}
        \begin{cases}
        \min_{a\in L^2([0,1],\R^d)}\ \ J(a) + \gamma  f(\xi(1),\Bar{g}), \\
        \dot{\xi}(s) = \sum_{i=1}^d a_i(s) U_i(\xi(s)) ,\ s\in[0,1], \\
        \xi(0) = (1,0,\dots,0)\in G^N(\mathbb{R}^d),
        \end{cases}
    \end{equation}
where the second term in the first line $f:G^N(\R^d)\times G^N(\R^d)\to \R$ is a non-negative continuous function such that $f(g,\bar g)>0$ if $g\neq \bar g$ and $f(\bar g, \bar g)=0$.
In a recent article \cite{scagliotti2023gradient} the author shows that in the limit $\gamma \rightarrow \infty$ the minima of the two optimization problems \eqref{gamma_opt_prob} and \eqref{ori_prob} coincide. 
In the next section, we report the $\Gamma$-convergence result of \cite{scagliotti2023gradient}, and some useful consequences that descend from it.
Generally, the main advantage of integrating the end-point cost into the objective is that we do not have to handle the final-time constraint $\xi(1)=\bar g$, which can be difficult to implement in practice.
The magnitude of $\gamma$ becomes a trade-off with finding the shortest path and arriving at the desired signature.

\subsection{Relaxing the endpoint constraint through $\Gamma$-convergence}\label{section_gamma_conv}

In this section, we take advantage of the results derived in \cite[Section~7]{scagliotti2023gradient} regarding the $\Gamma$-convergence of optimal control problems in $\R^d$.
More precisely, in \cite{scagliotti2023gradient} the author studied the problem of relaxing ``hard'' terminal-time constraints by adding to the cost of interest a weighted term, which was in charge of penalizing the deviation of the trajectories from the desired target state.
Since in \cite{scagliotti2023gradient} the focus was the study of gradient flows associated to optimal control problems, it was crucial to put a Hilbert structure on the space of admissible controls.
However, since the $\Gamma$-convergence argument does not really rely on the Hilbert structure, in this section we prefer to deal with a more general case, as we shall detail below. \\
The aim of this part is to provide a rigorous justification for the fact that taking $\gamma\to\infty$ in \eqref{gamma_opt_prob} leads to the same optimization problem as in \eqref{ori_prob}. 
This means the problem with regularization term weighted by $\gamma$ is an approximation on the original problem.  
We start by recalling the definition of $\Gamma$-convergence (for further details, we recommend the monograph \cite{D93}).

\begin{definition}[$\Gamma$-convergence]\label{gamma_conv}
    We say that a sequence of functionals $(\mathcal{F}_\gamma)_{\gamma \in \mathbb{R}}$ $\Gamma$-converges to $\mathcal{F}: L^p([0,1],\R^d) \to \mathbb{R}^+ \cup \{\infty\}$ w.r.t. weak topology of $L^p([0,1],\R^d)$ as $\gamma\rightarrow \infty$ if the following two conditions hold: 
    \begin{itemize}
        \item[1)] ($\liminf$-condition) for every sequence $(a_\gamma)_{\gamma \in \mathbb{R}^+} \subset L^p([0,1],\mathbb{R}^d)$ s.t. $a_\gamma \rightharpoonup_{L^p} a$ as $\gamma \rightarrow \infty$:
        \begin{equation*}
        \liminf_{\gamma \rightarrow \infty} \mathcal{F}_\gamma(a_\gamma) \geq \mathcal{F}(a).
        \end{equation*}
        \item[2)] ($\limsup$-condition) for every $a\in L^p([0,1],\mathbb{R}^d)$ there exists a so called recovery sequence which weakly converges $a_\gamma \rightharpoonup_{L^p} a$ as $\gamma \rightarrow \infty$ such that 
        \begin{equation*}
        \limsup_{\gamma \rightarrow \infty} \mathcal{F}_\gamma(a_\gamma) \leq \mathcal{F}(a).
        \end{equation*}
    \end{itemize}
\end{definition}

In our case, the functional from \eqref{gamma_opt_prob} takes the form: 
\begin{equation*}
    \mathcal{F}_\gamma(a) = J(a) + \gamma f(\xi(1),\bar{g}).
\end{equation*}
As mentioned at the beginning of this section, we revise the proof of the $\Gamma$-convergence step by step for our objective, and we extend the framework described in \cite{scagliotti2023gradient} to controls from $L^p([0,1],\R^d)$ with $p\in(1,\infty)$.
In this way, we can embrace as well more general costs of the form:
\begin{equation} \label{eq:gener_cost}
    \mathcal{F}_\gamma(a) = \| a \|_{L^p}^p + \gamma f(\xi(1),\bar{g}).
\end{equation}
\begin{remark}\label{rmk:gener_cost}
    In view of the main problem that we discuss in the present paper, i.e. the inversion of the truncated signature using a length-minimizer path, the correct exponent to consider in \eqref{eq:gener_cost} is $p=2$.
    Nevertheless, it can be interesting to extend our analysis to the case $p\in(1,\infty)$. In this case, recalling that $a_i=[\dot x]_i$ (see Theorem~\ref{control_sys}), we have that
    \begin{equation*}
    \mathcal{F}_\gamma(a) = 
    \int_0^1 |\dot x(t)|_2^p\, dt+ \gamma f(\xi(1),\bar{g}).
\end{equation*}
\end{remark}

We now present the $\Gamma$-convergence result.
Since the proof is quite technical, we postpone it to the Appendix.

\begin{theorem}\label{gamma_conv_thm}
    Given $\gamma>0$ we consider the control system \eqref{gamma_opt_prob} over the bounded domain $A_c=\{a\in L^p([0,1],\R^d), ||a||_{L^p}\leq c\}, \ c>0$. Then it holds on the one hand that the objective functionals $\mathcal{F}_\gamma:A_c \to \R^+$ admit a minimizer for each single $\gamma >0$. On the other hand, we can show that the sequence $(\mathcal{F}_\gamma)_{\gamma>0}$, $\Gamma$-converges w.r.t. the weak topology of $L^p$ to the functional
    \begin{equation}
    \label{eq:gamma_limit}
        \mathcal{F}(a) =
        \begin{cases}
            \|a\|_{L^p}^p, & \xi(1) = \bar{g},\\
            \infty , & \mbox{otherwise}.
        \end{cases}
    \end{equation}
\end{theorem}
\begin{proof}
    See the Appendix.
\end{proof}

\begin{remark} \label{rmk:bounded}
    It is important to note that the restriction to the set $A_c$ is purely due to technical reasons concerning $\Gamma$-convergence. Indeed, the topology underlying the $\Gamma$-convergence should be metrizable (or at least first countable). We recall that the weak topology of $L^p$ is not first countable (and hence, is not metrizable as well), unless it is restricted to a bounded set of $L^p$ (with $p\in(1,\infty)$).
    The parameter $c>0$ involved in the definition of the set $A_c$ (see also Step~2 in the proof) can be set as $c=\| \tilde a \|_{L^p}$, where $\tilde a \in \U$ is \emph{any} (not necessarily optimal) control such that the corresponding trajectory satisfies $\xi(1)=\bar g$, whose existence is guaranteed by Theorem~\ref{Chow}.
    Finally, we insist on the fact that, \emph{for every $\gamma>0$, every minimizer of $\mathcal{F}_\gamma$ is contained in $A_c$}. As a matter of fact, we have that
    \begin{equation*}
        \inf_{a\in A_c}\mathcal{F}_\gamma(a) = \inf_{a\in L^p([0,1],\mathbb{R}^d)} \mathcal{F}_\gamma(a).
    \end{equation*}
    Therefore, in practice, when addressing numerically the minimization of $\mathcal{F}_\gamma$, it is not necessary to take into account any kind of restriction for the controls.
\end{remark}

Note that this result connects the two optimization problems \eqref{ori_prob} and \eqref{gamma_opt_prob}, and hence it gives meaning to the minimizers of the relaxed optimization problem. More precisely, the minimizers of the relaxed optimization problem can be understood as approximations of the minimizer of the ``hard'' endpoint constraint problem \eqref{ori_prob}. From the $\Gamma$-convergence argument, it follows that the sequence of minimizers of $\mathcal{F}_\gamma$ denoted by $(a_\gamma)_{\gamma\in \mathbb{R}}$ is weakly pre-compact, and that any limiting point of this sequence minimizes $\mathcal{F}$. Nevertheless, in this setting it is also possible to prove strong convergence w.r.t. $L^p$-topology. 
\begin{corollary}
    Let $(a_\gamma)_\gamma$ be a sequence such that $a_\gamma \in \argmin \mathcal{F}_\gamma$ for every $\gamma$. 
    Then, $(a_\gamma)_\gamma$ is pre-compact with respect to the strong topology of $L^p$, and every limiting point as $\gamma\to\infty$ is a minimizer for $\mathcal{F}$.
\end{corollary}
\begin{proof}
From Remark~\ref{rmk:bounded} it follows that $(a_{\gamma})_\gamma$ is included in the bounded set $A_c\subset L^p$. 
Hence, we deduce that $(a_{\gamma})_\gamma$ is pre-compact in the \emph{weak} topology of $L^p$. Moreover, from a celebrated consequence of $\Gamma$-convergence (see \cite[Corollary~7.20]{D93}) it follows that every of its ($L^p$-weak) limiting point is a minimizer for $\mathcal{F}$.
We have to prove that the convergence of subsequences of $(a_\gamma)_\gamma$ towards minimizers of $\mathcal{F}$ holds as well in the \emph{strong} topology of $L^p$.
Let $(\gamma_m)_m$ be an increasing sequence such that $\gamma_m\to\infty$ and $a_{\gamma_m}\weak_{L^p} a$ as $m\to\infty$. We have to show that $a_{\gamma_m}\to_{L^p} a$.
We first observe that, since $a\in \argmin \mathcal{F}$, we have that $\mathcal{F}(a)<\infty$, and hence the corresponding curve of signatures fulfills $\xi(1)=\bar{g}$. Moreover, 
    \begin{align*}
        \mathcal{F}(a) =  \|a\|_{L^p}^p & = \|a\|_{L^p}^p + \gamma_m f(\xi(1),\bar{g})
        = \mathcal{F}_{\gamma_m}(a) \\
        &\geq \mathcal{F}_{\gamma_m}(a_{\gamma_m})  = \|a_{\gamma_m}\|_{L^p}^p + \gamma f(\xi_\gamma(1),\bar{g}) \\ 
        &\geq \| a_{\gamma_m}\|_{L^p}^p.
    \end{align*}
Thus, we deduce that 
    \begin{align}\label{ineq_apend}
        \|a\|_{L^p}^p \geq \limsup_{m\to \infty}  \|a_{\gamma_m}\|_{L^p}^p \geq \liminf_{m\to \infty}  \|a_{\gamma_m}\|_{L^p}^p \geq   \|a\|_{L^p}^p,
    \end{align}
where the last inequality is due to the lower semi-continuity of the $L^p$-norm for weakly convergent sequences. 
    From \eqref{ineq_apend} we get that $\|a_{\gamma_m}\|_{L^p}^p \rightarrow \|a\|_{L^p}^p$. 
    Paired with the weak convergence $a_{\gamma_m}\weak_{L^p} a$ this implies the strong convergence of $a_{\gamma_m}$ in $L^p$ by a classical result in functional analysis (\cite[Proposition~3.32]{brezis2011functional}), since $L^p$-spaces are uniformly convex for $1<p<\infty$. 
\end{proof}
Another immediate Corollary gives us a rate of decay to $0$ of the term responsible for enforcing softly the constraint. 
\begin{corollary}
    Let $a_\gamma$ be a sequence such that $a_\gamma\in \argmin \mathcal{F}_\gamma,\ \forall \gamma\geq 1$. Then, $f(\xi_\gamma(1),\bar{g}) = o(1/\gamma)$ as $\gamma\to\infty$. 
\end{corollary}
\begin{proof}
    In order to show that 
    $\lim_{\gamma\to\infty} \gamma f(\xi_\gamma(1),\bar{g})=0$, 
    we prove that, for every increasing sequence $(\gamma_m)_m$ such that $\gamma_m\to\infty$ as $m\to \infty$, the sequence $\big( \gamma_m f(\xi_{\gamma_m}(1),\bar{g})
    \big)_{m}$ admits a subsequence $\left( \gamma_{m_k}
        f(\xi_{\gamma_{m_k}}(1),\bar{g})
        \right)_{k}$
    such that 
    \begin{equation} \label{eq:lim_small_o}
        \lim_{k\to\infty} \gamma_{m_k}
        f(\xi_{\gamma_{m_k}}(1),\bar{g})=0.
    \end{equation}
    Let us extract from $a_{\gamma_m}$ a sub-sub sequence $a_{\gamma_{m_k}}$ such that $a_{\gamma_{m_k}}\rightharpoonup_{L^p} a,\ k\to \infty.$
    Repeating the same arguments as in the previous corollary, we can deduce that
    \begin{align*} \|a_{\gamma_{m_k}}\|_{L^p}^p \geq \lim_{k\to \infty} \|a_{\gamma_{m_k}}\|_{L^p}^p + \gamma_{m_k}f(\xi_{\gamma_{m_k}}(1),\bar{g}) \geq \lim_{k\to \infty} \|a_{\gamma_{m_k}}\|_{L^p}^p = \|a\|_{L^p}^p
    \end{align*}
    Hence, we obtain \eqref{eq:lim_small_o} and we conclude the proof. 
\end{proof}
\begin{remark}
    We observe that the previous Corollary strengthens an estimate proved in Step~3 of the proof of Theorem~\ref{gamma_conv_thm} (see the Appendix). Indeed, there it is observed that, for $\gamma$ large enough, we have $\gamma f(\xi_\gamma(1), \bar g)\leq C+1$, with $C>0$ constant, yielding $f(\xi_\gamma(1),\bar g) = O(1/\gamma)$.
\end{remark}
We have shown that the minimizers of $\mathcal{F}_\gamma$ are converging in the $L^p$ strong topology to the minimzers of $\mathcal{F}$. However, since $(\mathcal{F}_\gamma)$ are in general non-convex a natural question concerns the limits of local minimizers of $\mathcal{F}_\gamma$.
\begin{proposition}
    Assume that $(a_\gamma)_{\gamma>0}$ is a strongly convergent sequence $a_\gamma\to_{L^p} a$ as $\gamma\to\infty$ such that $a_{\gamma}$ is a local minimizer for $\mathcal{F}_\gamma$, i.e. there exists $r>0$ such that 
    \begin{equation*}
        \mathcal{F}_\gamma(a_\gamma)= \inf_{a'\in B_r(a_\gamma)}\mathcal{F}_\gamma(a') \leq C
    \end{equation*}
    for every $\gamma>0$.
    Then, $\mathcal{F}(a)=\inf_{a'\in B_{r/2}(a')} \mathcal{F}(a)$, i.e. $a$ is a local minimizer of $\mathcal{F}$.
\end{proposition}
\begin{proof}
    From $C\geq \mathcal{F}_\gamma(a_\gamma) = \|a_\gamma\|_{L^p}^p+\gamma f(\xi_\gamma(1),\bar{g}),$ we deduce that 
    \begin{equation*}
        f(\xi_\gamma(1),\bar{g})\leq \frac{1}{\gamma} \left(C-\|a_\gamma\|_{L^p}^p\right)
    \end{equation*}
    and, since $\gamma\to \infty$, we get $f(\xi_\gamma(1),\bar{g})\to 0$, which implies that $f(\xi_\infty(1),\bar{g})= 0$, where $\xi_\infty$ is the curve in $G^N(\R^d)$ corresponding to the limiting control $a$. This in turn implies that $\xi_\infty(1)=\bar g$, i.e., $\mathcal{F}(a)= \|a\|_{L^p}^p$.\\
    Let us choose $\bar{\gamma}\geq 0$ such that $\|a_\gamma-a\|_{L^p} \leq r/2$ for every $\gamma \geq \bar \gamma$.
    Next assume that there exists $a'\in B_{r/2}(a)$  such that $\mathcal{F}(a')\leq \mathcal{F}(a)$. Then, recalling that $B_{r/2}(a)\subset B_r(a_\gamma)$ and that $\mathcal{F}_\gamma(a_\gamma)= \inf_{ B_r(a_\gamma)}\mathcal{F}_\gamma$, we deduce
    \begin{equation*}
        \mathcal{F}(a)=\|a\|_{L^p}^p \geq F(a')=\|a'\|_{L^p}^p\geq \mathcal{F}_\gamma(a_\gamma)\geq \|a_\gamma\|_{L^p}^p.
    \end{equation*}
    Since $a_\gamma\to_{L^p} a$, we have that $\|a_\gamma\|_{L^p}^p \rightarrow \|a\|_{L^p}^p$, which yields
    \begin{align*}
        \|a\|_{L^p}^p \geq \|a'\|_{L^p}^p\geq \lim_{\gamma\to \infty}  \|a_\gamma\|_{L^p}^p = \|a\|_{L^p}^p,
    \end{align*}
    from which we read $\mathcal{F}(a) = \|a\|_{L^p}^p = \|a'\|_{L^p}^p=\mathcal{F}(a')$. 
    Therefore, we have shown that, for all $a'\in B_{r/2}(a)$ such that $\mathcal{F}(a')\leq \mathcal{F}(a)$, we get $\mathcal{F}(a')=\mathcal{F}(a)$. This concludes the proof.
\end{proof}
The arguments above show that we can address problem \eqref{gamma_opt_prob}, that is more suitable for numerical purposes, since it incorporates a \emph{softer} formulation of the terminal-time constraint. 
In the limit $\gamma \rightarrow \infty$, we recover the solutions original hard-constrained problem \eqref{ori_prob}.

\subsection{Pontryagin's Maximum Principle for signatures}

The sections before clarified why it is useful to look at problem \eqref{gamma_opt_prob} instead of \eqref{ori_prob}. Therefore, we state the Pontryagin Maximum Principle (PMP) for our problem, and we introduce the iterative algorithm proposed in \cite{sakawa1980} and based on the PMP. In \cite{sakawa1980} the authors proved that the method provides a sequence of controls along which the cost functional is monotonically decreasing. 
This algorithm was designed to solve optimal control problems of the same type as the ones that were formulated in the previous section. We report that it has been recently applied for linear-control problems concerning the approximation of diffeomorphisms \cite{scagliotti2023deep} and the optimal simultaneous steering of an ensemble of systems \cite{scagliotti2023optimal}.
We recall that the PMP provides a set of necessary conditions for the optimality of the control.
For the purposes of our paper, we can rely on a classical formulation of the PMP, as, for example, the one described in the textbook \cite{bressan}. 
Before stating the PMP we have to prepare the setting for Algorithm \ref{alg:loop} and introduce a reformulation of the objective in \eqref{gamma_opt_prob}.
\begin{equation} \label{eq:ctrl_problem_PMP}
    \begin{split}
    &\min_{a \in L^2([0,1]\mathbb{R}^d)}\ \ J(a) + \gamma  f(\xi(1),\Bar{g})  \\
    &\Leftrightarrow \min_{a \in L^2([0,1]\mathbb{R}^d)}\ \ \frac{1}{2} \int_0^1 |a(t)|^2 \, dt + \gamma  f(\xi(1),\Bar{g})  \\
    &\Leftrightarrow \min_{a \in L^2([0,1]\mathbb{R}^d)}\ \ \int_0^1 \left( \frac{1}{2} |a(t)|^2 + \gamma  \frac{\partial}{\partial t} f(\xi(t),\Bar{g}) \right)  \, dt
    \end{split}
\end{equation}
For this objective, we state the Pontryagin Maximum Principle below.

\begin{theorem}[Pontryagin's Maximum Principle (PMP)]\label{PMP}
Given the optimal control problem \eqref{gamma_opt_prob}, we define the Hamiltonian $H:G^N(\R^d)\times \R^d \times T^N(\R^d) \to \R$ as the following function
    \begin{equation}\label{Hamiltonian}
        H(\xi,\omega,p) =  \frac{1}{2} |\omega|^2  + 
        \langle \gamma  
        \nabla_{\xi} f(\xi,\Bar{g}) + p, \sum_{i=1}^d \omega_i U_i(\xi) \rangle,
    \end{equation}
where $p\in T^N(\mathbb{R}^d)=\bigoplus_{k=1}^N (\mathbb{R}^d)^{\otimes k }$ is called the \emph{costate variable}. Assume that the trajectory $s\mapsto \xi^*(s) \in C([0,1],G^N(\mathbb{R}^d))$ corresponds to the curve of optimal controls $s\mapsto [\dot{x}^*]_i(s) := a_i^*(s) \in L^2([0,1],\mathbb{R}^d), i\in \{0,\dots,d\}$, and let us consider the solution of the following linear ODE (\emph{adjoint equation}):
\begin{equation} \label{eq:costate_backward}
    \begin{cases}
    \dot p(s) = -\frac{\partial}{\partial \xi} H(\xi^*(s),a^*(s),p(s)) & \mbox{a.e. in }[0,1],\\
    p(1) = 0\in T^N(\mathbb{R}^d).
    \end{cases}
\end{equation}
Then, the necessary optimality condition reads as
    \begin{equation} \label{eq:max_condition}
        H(\xi^*(s),a^*(s),p(s)) = \min_{\omega \in \R^d} H(\xi^*(s),\omega,p(s))
    \end{equation}
for a.e. $s\in[0,1].$
\end{theorem}
\begin{proof}
    See \cite[Theorem~6.1.1]{bressan}.
\end{proof}

\begin{remark}\label{rmk:abnormal}
    In Theorem~\ref{PMP}, we formulated the PMP for normal extremals only. Indeed, the optimal control problem as formulated in \eqref{eq:ctrl_problem_PMP} does not admit abnormal extremals nor singular arcs.
\end{remark}


We observe that the PMP stated in Theorem~\ref{PMP} requires the solution of two separate differential equations.
Namely, we have the \emph{forward} controlled dynamics $\dot \xi(s) = \sum_{i=1}^d a_i(s)U_i(\xi(s))$ (with the Cauchy datum prescribed at the initial instant), and the \emph{backward} linear ODE \eqref{eq:costate_backward} for the costate $p$, where the Cauchy condition is assigned at the terminal instant.
Finally, the relation \eqref{eq:max_condition} relates the optimal trajectory $\xi^*$, the adjoint trajectory and the optimal control $a^*$.
The algorithm analyzed in \cite{sakawa1980} exploits these relations, and we describe below the implementation of this procedure to our specific problem.
The mathematical details for each single step are described in the next section, where we discuss how we solve each of the differential equations numerically.\\

\begin{algorithm}[H]
\SetKwInOut{Input}{Input}\SetKwInOut{Output}{Output}
\Input{$a^0 \in L^2([0,1],\mathbb{R}^d)$ (initial guess for the control)\\
$\xi^0\in \mathbb{X}$ (initial trajectory of signatures) \\
$D\in \mathbb{N}$ (number of discretization points)\\
$M\in\mathbb{N}$ (number of iterations of the algorithm)\\
$C\in \mathbb{R}^+$ (used in the augmented Hamiltonian)}
\Output{$a^* \in L^2([0,1],\mathbb{R}^d)$ (approximation of the optimal control)\\
$\xi^*\in \mathbb{X}$ (optimal trajectory of signatures) }
\For{j in $\{1,\dots,M\}$}{
    1) Solve backward for the equidistant grid $s\in \{1,\frac{D-1}{D},\dots,\frac{1}{D},0\}$ \begin{equation*}\dot p^{j-1}(s) = -\frac{\partial}{\partial g} H(\xi^{j-1}(s),a^{j-1}(s),p^{j-1}(s)),\ \ p^{j-1}(1) = 0\in T^N(\mathbb{R}^d)\end{equation*}

    2) Find $a^j(s)$ and $\xi^j(s)$ for $s\in \{0,\frac{1}{D},\dots,1\}$ by alternating the following two steps:
    
    2.1) Augment the Hamiltonian by a quadratic term:
                    \begin{equation*} 
                    \begin{split}
                    Q(\xi^{j}(s),a^{j}(s),p^{j-1}(s),C,a^{j-1}(s))=\min_{\omega\in \mathbb{R}^d}
                    &\big[
                    H(\xi^{j}(s),\omega,p^{j-1}(s)) \\
                    & \, + (\omega-a^{j-1}(s))'C(\omega-a^{j-1}(s))
                    \big]
                    \end{split}
                    \end{equation*}
                Solve $\nabla_\omega Q(\xi^{j}(s),\omega,p^{j-1}(s),C,a^{j-1}(s)) = 0$ for $\omega$ $\Rightarrow$ set $a^j(s) = \omega$ 
                
    2.2) Use Chen's Relation in Theorem \ref{Chen} to solve  
                \begin{equation*}\dot{\xi}^j(s) = \sum_{i=1}^d a_i^j(s) U_i(\xi^j(s)),\ \ \xi(0)=(1,0,\dots,0)\in G^N(\mathbb{R}^d)\end{equation*}

    3.)  Check, whether our algorithm improves: \\
    \If{$J(a^{j-1}) + \gamma  f(\xi^{j-1}(1),\Bar{g})>J(a^{j}) + \gamma  f(\xi^{j}(1),\Bar{g})$}{continue}
    \Else{$a^{j-1} \rightarrow a^{j}$ and $C = 2C$}
    $\xi^* = \xi^j, \ a^* = a^j$}
\KwRet{$\xi^*,a^*$}
\caption{Iterative PMP scheme}
\label{alg:loop}
\end{algorithm}

\subsection{Finding the Shortest Path}\label{find_path}

In this section, we explain the mathematical details of each step in the algorithm introduced in the previous section. 
In order to softly enforce that $\xi(1)$ is close to $\bar g$, we take as terminal cost
    \begin{equation}\label{regularizer_lip}
        f(\xi(1),\bar{g}) = \sqrt{1+|\xi(1)-\bar{g}|^2}-1 = \sqrt{ 1+ \sum_{k=0}^N \sum_{|l|=k} |\xi_l(1)-\Bar{g}_l|^2}-1,
    \end{equation}
where the multi-index set is given by $\{|l|=k\}=\{i_1\dots i_k|i_1,\dots,i_k \in \{0,1,\dots,d\}\}$. 
This cost compares the square of each entry of $\bar{g}$, with the signature $\xi(1)$, i.e. where the trajectory of the states actually arrives. 
In the Hamiltonain \eqref{Hamiltonian} this end point constraint appeared with the derivative w.r.t. $s\in[0,1]$.
For this reason, we compute:
\begin{equation} \label{time_derivative}
    \begin{split}
    \frac{\partial}{\partial s} f(\xi(s),\bar{g}) &= \frac{\partial}{\partial s} \sqrt{||\xi(s)-\bar{g}||^2+1}-1 \\
    &= \frac{\sum_{k=0}^N \sum_{|l|=k} (\xi_l(s)-\Bar{g}_l) (\dot{\xi}(s))_l }{\sqrt{||\xi(s)-\Bar{g}||^2+1}}   \\
    &=  \frac{\sum_{k=0}^N \sum_{|l|=k}  (\xi_l(s)-\Bar{g}_l) a_{i_k}(s) \xi_{l \setminus i_k}(s)}{\sqrt{||\xi(s)-\Bar{g}||^2+1}},  
    \end{split}
\end{equation}
where we set $(l \setminus i_k) := i_1 \dots i_{k-1}$, and we used $\dot \xi(s) = \sum_{i=1}^d a_i(s)\xi(s) \otimes e_i$. 
This derivation is employed by the algorithm for solving the differential equation \eqref{eq:costate_backward} for the costates $s\mapsto p(s)$. 
Since only the value at the terminal time is given, we have to solve the differential equation backward, and we implement the backward implicit Euler scheme at each step $s\in \{1,\frac{D-1}{D},\dots,\frac{1}{D},0\}$. 
In addition, we have to distinguish three cases, since the derivative of the Hamiltonian w.r.t. the $\xi$ variable has three different expressions according to the different levels of the signature. 
We consider $\frac{\partial}{\partial \xi_{i_1,\dots,i_k}} H(\xi,a,p)$ for the following three different cases:
\begin{equation*}
\begin{cases}
    \text{a)} \ \ k\in\{2,\dots,N-1\}\\\text{b)}\ \ k=1 \\\text{c)}\ \ k=N
\end{cases}
\end{equation*}
First of all we compute the following derivatives for each of these three cases:
\begin{itemize}
    \item[a)] For $k\in \{2,\dots,N-1\}$ we use the multiindex $j$ such that $|j|=k$ and $j_i$ is the $i$-th component of $j$:
    \begin{align*}
        & \frac{\partial}{\partial \xi_{j}} H(\xi,a,p) \\
        &\stackrel{\eqref{Hamiltonian}-\eqref{time_derivative}}{=} \frac{\partial}{\partial \xi_{j}} \bigg(\gamma \frac{\sum_{k=0}^N \sum_{|l|=k}  (\xi_l-\Bar{g}_l) a_{j_k} \xi_{l \setminus j_k}}{\sqrt{|\xi(s)-\Bar{g}|^2+1}} +  \langle p, \sum_{i=1}^d a_i \xi \otimes e_i \rangle \bigg)\\
        &=  \gamma \bigg[\bigg(a_{j_k} \xi_{j\setminus j_k} + \sum_{i_{k+1}=1}^d (\xi_{ji_{k+1}}-  \Bar{g}_{ji_{k+1}}) a_{i_{k+1}} \bigg) \frac{1}{\sqrt{|\xi-\bar{g}|^2+1}} \\ 
        &\qquad + \sum_{k=0}^N \sum_{|l|=k}  (\xi_l-\Bar{g}_l) a_{l_k} \xi_{l \setminus l_k} \frac{-(\xi_j-\Bar{g}_j)}{\big(|\xi-\bar{g}|^2+1\big)^{\frac{3}{2}}} \bigg] \\
        &\qquad+ \sum_{i_{k+1}=1}^d   p_{ji_{k+1}} a_{i_{k+1}}.
    \end{align*}

    \item[b)] For $k=1$ we have only one index $j\in\{1,\dots,d\}$. The derivation stays exactly the same until line 3 in part a), from line 4 we use that $g_0 = 1$: 
     \begin{align*}
         \frac{\partial}{\partial \xi_{j}} H(\xi,a,p) &=  \gamma \bigg[\bigg(a_{j} + \sum_{i_2=1}^d (\xi_{ji_{2}}-  \Bar{g}_{ji_{2}}) a_{i_{2}} \bigg)\frac{1}{\sqrt{|\xi-\bar{g}|^2+1}} \\ 
        &\qquad+ \sum_{k=0}^N \sum_{|l|=k}  (\xi_l-\Bar{g}_l) a_{l_k} \xi_{l \setminus l_k} \frac{-(\xi_j-\Bar{g}_j)}{\big(|\xi-\bar{g}|^2+1\big)^{\frac{3}{2}}} \bigg]
        \\&\qquad 
        + \sum_{i_2=1}^d   p_{ji_{k+1}} a_{i_2}.
     \end{align*}
     
    \item[c)] For $k=N$ we call the corresponding multi-index $j$ such that $|j|=N$ and we start again from the first line of part b). In this case, we have to take into account that in $G^N(\mathbb{R}^d)$ the signatures do not have higher orders than $N$, neither have the costates $p_j \in T^N(\mathbb{R}^d)$. 
    Hence, the last line from part a) reads:
    \begin{align*}
        \frac{\partial}{\partial \xi_{j}} H(\xi,a,p) &= \gamma \bigg(a_{j_N} \xi_{j\setminus j_N} \frac{1}{|\xi-\bar{g}|^2+1} \\ 
        &\qquad + \sum_{k=0}^N \sum_{|l|=k}  (\xi_l-\Bar{g}_l) a_{i_k} \xi_{l \setminus i_k} \frac{-(\xi_j-\Bar{g}_j)}{\big(|\xi-\bar{g}|^2+1\big)^{\frac{3}{2}}} \bigg).
    \end{align*}
\end{itemize}

\subsubsection*{Approximation of the costate equation}
Given these derivatives of the Hamiltonian $H$, we can proceed with the numerical resolution of the backward differential equation \eqref{eq:costate_backward}, starting with the final value $p(1) = 0\in G^N(\mathbb{R}^d)$. 
We employ the implicit Euler approach to approximate the ODE on the discrete grid $\{0, \frac{1}{D},\dots, \frac{D-1}{D}, 1\}$, and, when  $t \in \{0, \frac{1}{D},\dots, \frac{D-1}{D}, 1\}$, we consider the following approximation: 
\begin{equation*}
    \mbox{for } s\in [t,t+1/D ), \quad \dot{p}(s) \simeq \frac{p(t+1/D)-p(t)}{\Delta t} \simeq -\frac{\partial}{\partial \xi} H(\xi(t),a(t),p(t)),
\end{equation*}
which yields
\begin{equation} \label{eq:update_p}
    p(t) = p(t+ 1/D ) + \frac1D \frac{\partial}{\partial \xi} H(\xi(t),a(t),p(t)).
\end{equation}
Recalling that the linear ODE that defines $s\mapsto p(s)$ has a \emph{backward evolution}, the relation \eqref{eq:update_p} provides an implicit Euler scheme for the approximate computation of $p$ in the nodes.
We observe that, when computing $p$ at the node $t$, we already have access to the values $\xi(t)$ and $a(t)$, but the unknown $p(t)$ appears on both sides of the equation. 
In the implementation, we solve this implicit definition by means of fixed-point iterations.
The number of fixed point iterations is a hyperparameter, however, this procedure can be stopped whenever the values of two succeeding iterations are sufficiently close. 
\noindent
This gives us a way of finding a trajectory of the costates which were given by the backward differential equation in Step~1) of Algorithm~\ref{alg:loop}.

\subsubsection*{Update of  control and trajectory}

We now analyze the second step of Algorithm~\ref{alg:loop}, where the current guess for the controls $s\mapsto a^*(s) \in L^2([0,1],\mathbb{R}^d)$ is updated, as well as the corresponding trajectory in the space of signatures $s \mapsto \xi^*(s) \in \mathbb{X}$. The updates of $a$ and $\xi$ take place simultaneously, relying on the costate trajectory approximated in the previous step. \\ 
We insist on the fact that the interesting aspect is that, at each iteration, $a$ and $\xi$ should be derived simultaneously, and that there is an interdependence between them.
Hence, we exploit the fact that we have an initial value for the $\xi$ trajectory, which is $\xi^j(0)=(1,0,\dots,0)\in G^N(\mathbb{R}^d)$ for every iteration $j=1,\ldots, M$ of the algorithm. 
Secondly, we observe that the values of the control $a^j$ at the time $t\in \{0,\frac{1}{D},\dots,1\}$ is responsible for the change in the trajectory  $\xi^j$ from $t$ to $t+1/D$. 
Finally, the value $a^j(t)$ is related to $\xi^j(t)$ through the maximization of the augmented Hamiltonian function:
\begin{equation} \label{eq:augm_ham}
    a^j(t) = \argmin_{\omega \in \mathbb{R}^d} \Big[
    H(\xi^{j}(t),\omega,p^{j-1}(t)) + (\omega-a^{j-1}(t))'C(\omega-a^{j-1}(t)) \Big],
\end{equation}
where $p^{j-1}$ is the costate trajectory approximated at the previous step of Algorithm~\ref{alg:loop}, and the second term is a memory term that is somehow reminiscent of the \emph{minimizing movement scheme} (see, e.g., \cite{santambrogio2017euclidean}) and it penalizes too large deviations of $a^j(t)$ from $a^{j-1}(t)$. In other words, the constant $C>0$ plays here the role of the \emph{inverse of the learning rate} for the update of the controls.
To sum up, for every $t\in \{0,1/D,\ldots, 1\}$, to compute $a^j(t)$ we need $p^{j-1}(t)$, $\xi^j(t)$ and $a^{j-1}(t)$. Once $a^j(t)$ has been obtained, with $\xi^j(t)$ and $a^j(t)$ at hand, we obtain $\xi^{j}(t+1/D)$.
We try to visualize this procedure in Figure \ref{fig:Sakawa}.\\
We finally describe the computation of $\xi^j(t+1/D)$. For this evolution we can exploit Chen's relation introduced in Theorem~\ref{Chen}. More precisely, we consider the affine path $\tilde{x}:[0,1/D)\rightarrow \mathbb{R}^d$ such that $\tilde x(0) = 0$ and $\dot{\tilde x}(s) = a^j(t)$ for every $s\in [0,1/D]$.
Next, we concatenate the path  $\Tilde{x}:[0,1/D) \to \mathbb{R}^d$ with $x:[0,t]\to \mathbb{R}^d$ that was inductively generated up to time $t$ in the same manner:
\begin{equation*}
    x\sqcup \Tilde{x} (s) = \begin{cases}
        x(s) & s\in[0,t],\\
        \Tilde{x}(s)- \tilde x(0)+ {x}(t) & s\in[t,t+1/D].
    \end{cases}
\end{equation*}
As mentioned, due to Chen's relation we have $S_N(x\sqcup \Tilde{x}) = S_N(x)\otimes S_N(\Tilde{x})$, so that we can either derive the signature of the concatenated path directly or, equivalently but more efficiently, we can merge the paths on the level of signatures via a tensor product. This procedure leads to a signature which is the next point on the signature trajectory.
More precisely, we carry out the step 2.2) in the algorithm as 
\begin{equation*}
    \xi^j (t+1/D) = S_N(x\sqcup \tilde x) = \xi^j(t) \otimes S_N(\tilde x).
\end{equation*}
At this point, we are in position to generate $a^j(t+1/D)$.
This alternating procedure as visualized in Figure~\ref{fig:Sakawa} leads to the simultaneous construction of the control and the state trajectory.

\begin{figure}[h!]
	\centering
	\begin{tikzpicture}[shorten >=1pt]
		\tikzstyle{unit}=[draw,shape=circle,minimum size=1.15cm]
		\tikzstyle{hidden}=[draw,shape=circle,minimum size=1.15cm]
 
		\node[unit](x_1) at (0,4){$p^{j-1}(0)$};
		\node[unit](x_2) at (0,2){$a^{j-1}(0)$};
		\node[unit](x_3) at (0,0){$\xi^j(0)$};
 
		\node[hidden](a_1) at (2,2){$a^j(0)$};

		\node[unit](x_33) at (3.5,0){$\xi^j(t_1)$};

		\node(d3) at (11,0){$\ldots$};
		\node(d2) at (11,2){$\ldots$};
		\node(d1) at (11,4){$\ldots$};

		\node[unit](hl0) at (5.1,4){$p^{j-1}(t_1)$};
		\node[unit](hl1) at (5.1,2){$a^{j-1}(t_1)$};
 
		\node[hidden](a_2) at (8,2){$a^j(t_1)$};

		\node[unit](hl33) at (9.5,0){$\xi^j(t_2)$};

		\draw[->] (x_1) -- (a_1);
		\draw[->] (x_2) -- (a_1);
		\draw[->] (x_3) -- (a_1);
 
		\draw[->] (x_3) -- (x_33);
            \draw[->] (a_1) -- (x_33);

            \draw[->] (hl0) -- (a_2);
		\draw[->] (hl1) -- (a_2);
 
		\draw[->] (a_2) -- (hl33);

            \draw[->] (x_33) -- (hl33);
            \draw[->] (x_33) -- (a_2);

		\draw [decorate,decoration={brace,amplitude=8pt},xshift=-4pt,yshift=0pt] (-0.5,4.8) -- (0.75,4.8) node [black,midway,yshift=+0.6cm]{$t=0$};
            \draw node [black,midway,yshift=+5.8cm]{Given};
            \draw [decorate,decoration={brace,amplitude=8pt},xshift=-4pt,yshift=0pt] (4.6,4.8) -- (5.85,4.8) node [black,midway,yshift=+0.6cm]{$t=t_1$};
            \draw node [black,midway,xshift=5.05cm,yshift=+5.8cm]{Given};
		\draw [decorate,decoration={brace,amplitude=8pt},xshift=-4pt,yshift=0pt] (1.5,3) -- (2.75,3) node [black,midway,yshift=+0.6cm]{Step 2.1};
		\draw [decorate,decoration={brace,amplitude=8pt},xshift=-4pt,yshift=0pt] (3,1) -- (4.25,1) node [black,midway,yshift=+0.6cm]{Step 2.2};
          \draw [decorate,decoration={brace,amplitude=8pt},xshift=-4pt,yshift=0pt] (7.6,3) -- (8.85,3) node [black,midway,yshift=+0.6cm]{Step 2.1};
        \draw [decorate,decoration={brace,amplitude=8pt},xshift=-4pt,yshift=0pt] (9,1) -- (10.25,1) node [black,midway,yshift=+0.6cm]{Step 2.2};
	\end{tikzpicture}
	\caption{First two steps of the Algorithm \ref{alg:loop}}
	\label{fig:Sakawa}
\end{figure}
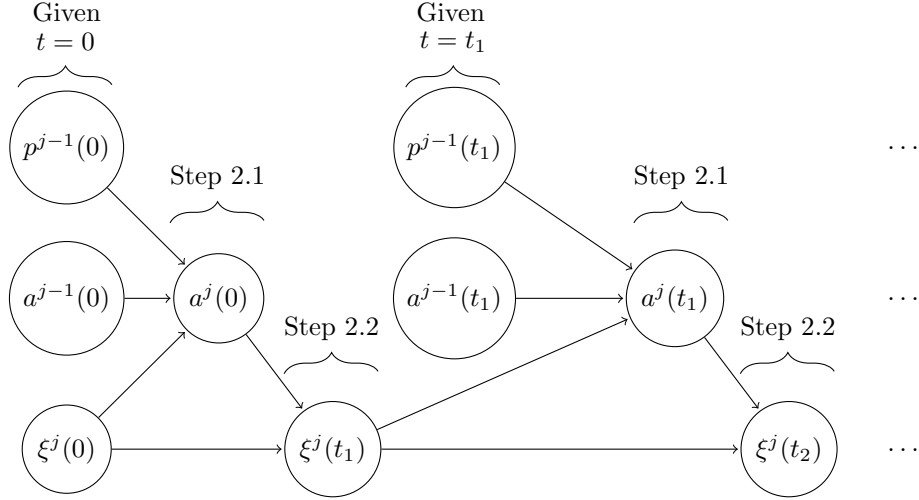
%
%
%
%
%
 
\subsubsection*{Validation of the update}

Finally, for step 3) in Algirhtm~\ref{alg:loop}, we compute the cost functional on the updated control/trajectory couple $(a^j,\xi^j)$, and we compare it with the one corresponding to the previous iteration $(a^{j-1},\xi^{j-1})$.
In the case of a decrease, we accept the update, otherwise we reject it and we increase the value of the constant $C>0$ involved in the definition of the augmented Hamiltonian \eqref{eq:augm_ham}.

\vspace{11pt}

In conclusion, we have adapted the algorithm from \cite{sakawa1980} in order to find the shortest path $s\mapsto \xi(s)\in \mathbb{X}$ connecting the neutral element $(1,0,\dots,0)\in G^N(\mathbb{R}^d)$ to a given signature $\bar{g}\in G^N(\mathbb{R}^d)$. 
On the one hand, the $\Gamma$-convergence result provided a theoretical foundation of the problem we consider. On the other hand, we could exploit the relation between the controls and the signatures as well as the algebraic properties of the signatures to set up the numerical scheme we described in this section. 
The next section deals with a version of the optimal control problem which leaves the final time variable such that we have to find the optimal terminal time as well.

\subsection{Version with Final Variable Time}\label{sec:fin_var_time}

From here we consider a reformulation of the problem, which is equivalent but turns into a final variable-time problem. 
Despite being equivalent, this formulation of the problem is more suitable for the numerical simulations, since we do not have to deal with the parameter $\gamma$. Indeed, in order to be in the regime of the $\Gamma$-convergence result, $\gamma$ should be taken sufficiently large.
Before we describe this method, we recall that the energy functional is not invariant under reparameterizations. For this reason, we have to consider the equivalent problem of minimizing the length under the additional constraint that the speed is constant as described in Theorem \ref{equivalence_of_probs}.

Moreover, choosing the right final time is the challenge one has to solve in this case. First of all, we use the following reparameterization of time in the objective
    \begin{equation*}dz = |a(s)|ds\end{equation*}
such that $z(1) = \int_0^1|a(s)|ds$ and $z(0) = 0$ and 
    \begin{equation}\label{reparam_obj}
        \ell (a) = \frac{1}{2} \int_0^1 |a(s)| ds = \frac{1}{2} \int_0^{z(1)} dz.
    \end{equation}
This means that the objective reduces to the following time-optimal problem:
    \begin{equation}\label{time_prob}
        \begin{cases}
        \bar T:= \min \ T, \\
        \dot{\xi}(s) = \sum_{i=1}^d a_i(s) U_i(\xi(s)) ,\
        |a(s)|_2 = 1
        , \ s\in[0,T], \\
        \xi(0) = (1,0,\dots,0)\in G^N(\R^d), \\
        \xi(T) = \bar{g}\in G^N(\R^d),
        \end{cases}
    \end{equation}
Since we do not know the final correct time a priori, we suggest the following simple strategy to approximate it.
To this end, we notice that for a given time $T^*$ there are three different scenarios that can happen and can be checked in the numerics. In what follows, the function $f:G^N(\R^d)\times G^N(\R^d) \to \R$ is the one that has been defined in \eqref{regularizer_lip}.
\begin{itemize}
    \item[S1)] Final time too short ($T^*< \bar T$). In this case, we have that
    \begin{equation*}
    f(\xi(T^*),\Bar{g}) > 0        
    \end{equation*}
    for every curve $\xi:[0,T^*]\to G^N(\R^d)$ satisfying $\dot{\xi}(s) = \sum_{i=1}^d a_i(s) U_i(\xi(s))$, with $|a(s)|_2 = 1 $ for $s\in[0,T]$, and $\xi(0)=(1,0,\ldots,0)$.
    
    \item[S2)] Final time correct ($T^* = \bar T$). In this case, the minimum is exactly attained and we can recover the optimal controls, i.e., there exists $a^*:[0,\bar T]\to \R^d$ and $\xi:[0,\bar T]\to G^N(\R^d)$ such that
    \begin{equation*}
        f(\xi(\bar T),\Bar{g}) = 0,        
    \end{equation*}
    and $\dot{\xi}^*(s) = \sum_{i=1}^d a_i^*(s) U_i(\xi(s))$, with $|a^*(s)|_2 = 1 $ for $s\in[0,\bar T]$, and $\xi^*(0)=(1,0,\ldots,0)$.
    
    \item[S3)] Final time too large ($T^*> \bar T$). In this case, the target terminal point can be achieved as in S2), however, the curves that do that are no more necessarily geodesics. 
\end{itemize}  
Therefore, we first guess the final time $T^*$ and switch to the same setting of finding the optimal controls directly as in the section before. 
Analogously, we try to simplify the optimization problem and we introduce a problem with terminal-time penalization, that softens the constraint $\xi(\bar{T})=\bar{g}$, as we did in \eqref{gamma_opt_prob}. 
However, this version does not deal with a trade-off between two different parts of the objective. Thus, the parameter $\gamma$ disappears, and we just have to minimize the terminal-cost, while using our guess $T^*$:
    \begin{equation*}\min_{a \in L^2([0,T^*],\mathbb{R}^d)} f(\xi({T^*}),\Bar{g}).\end{equation*}
Next, we shift the problem into a running cost problem in order to set up the prerequisites for Algorithm \ref{alg:loop}. Accordingly, the control problem takes the following form:
    \begin{equation*}
        \begin{cases}
            \min_{a \in L^2([0,T^*],\mathbb{R}^d)}\  \int_0^{T^*} \frac{\partial}{\partial s} f(\xi(s),\Bar{g}) ds\\
\dot{\xi}(s) = \sum_{i=1}^d a_i(s) U_i(\xi(s)) ,\
        |a(s)|_2 = 1
        , \ s\in[0,T^*], \\
        \xi(0) = (1,0,\dots,0)\in G^N(\R^d),
        \end{cases}
    \end{equation*} 
For the optimization routine we follow essentially the same steps as in the previous sections, whereas the simplified objective leads to slightly different derivatives in Step 1) and 2) of Algorithm \ref{alg:loop}. Below we list the commonalities and differences:  
\begin{itemize}
    \item Again we chose the same Lipschitz continuous regularizer from \eqref{regularizer_lip}.
    \item In step 1) of the algorithm the derivation of the $p$ trajectory via the backward equation only depends on the derivative of the Hamiltonian w.r.t. to the $g$ component. Hence, the first part of the objective in \eqref{gamma_opt_prob} does not play a role on the procedure and can be derived in the exact same manner for the final variable time section. The only exception is that we leave out the $\gamma$ parameter as described above. 
    \item In step 2) of the algorithm when we derive the trajectory of the controls via the augmented Hamiltonian in the same way, but the Hamiltonian does not contain the first term that appears at the right-hand side of \eqref{Hamiltonian}. Namely, we have that
    \begin{equation*}
        H(\xi,\omega,p) = 
        \langle  
        \nabla_{\xi} f(\xi,\Bar{g}) + p, \sum_{i=1}^d \omega_i U_i(\xi) \rangle
    \end{equation*}
    \item In Step 3) of the algorithm we take the objective functional described in this section and compare the evaluations on the trajectories of the actual and the previous iteration. Except of the objective this procedure stayed the same as in the original version before.  
\end{itemize}
Having derived the optimal controls by this method we can decide whether we increase or decrease our initial guess for the final time ${T^*}$ by the following decision rule. Determine an error tolerance $\epsilon>0$:
\begin{itemize}
    \item[A1)] If the error $f(\xi({T^*}),\Bar{g})>\epsilon$, we are in scenario S1). Thus, we have to increase the initial guess ${T^*}$.
    \item[A2)] If the error $f(\xi({T^*}),\Bar{g})\leq \epsilon$ we are either in Scenario S2) or S3), thus choose a smaller time ${T^*}$ :  
    \begin{equation*}\begin{cases}
        \text{if the length decreases we were in S3) and thus, we can choose a smaller $T^*$}\\
        \text{if the error increases we were (approximately) in S2) and already optimal}
    \end{cases}\end{equation*}
\end{itemize}
In order to find the optimal time such that $f(\xi(T^*),\Bar{g})\leq \epsilon$ and at the same time the optimal controls correspond to the shortest path, we suggest to start with a low initial time ${T^*}$. 
Then we iteratively increase the initial guess according to $A1)$ and follow the adapted Algorithm~\ref{alg:loop} to find the optimal controls for each iteration, until we end up in the case of $A2)$. 
From here we iteratively decrease the time $T^*$ at a much smaller scale than in the case $A1)$, until we reach $f(\xi({T^*}),\Bar{g})\leq \epsilon$ and the length increases for every further reduction for the level of $T^*$. 

In conclusion, in this section we have introduced an equivalent version of the optimal control problem which has final variable time. 
More precisely, given the correct final time we can find an exact solution to our optimal control problem. Nevertheless, the drawback of this method is that the correct final time is actually an unknown of the problem.

\subsection{Version with Additional Vector Fields}\label{sec:add_vec}

In this section, we present an extension to the first algorithm without the final variable time. Here, we integrate additional vector fields related to the brackets introduced in Definition~\ref{brackets}. 
As we recalled in Definition~\ref{lie_algebra}, the Lie algebra is generally spanned by the Lie brackets. Further, we noted in Remark~\ref{sub_riemann} that the tangent space at a single element in $G^N(\R^d)$ coincides with the Lie algebra generated by the vector fields $\xi\mapsto U_i(\xi) = \xi \otimes e_i$, $ i=1,\ldots,d$. 
Thus, the curves of signatures solving the control system from Theorem \ref{control_sys} can reach every possible signature in $G^N(\R^d)$, provided that we find appropriate weights in front of the vector fields. 
Nevertheless, a natural question is whether the directions of additional vector fields increase the flexibility of the control system and thus, lead to shorter curves. Thus, we can consider the extended control system including vector fields related to the single Lie brackets in the Lie algebra. For simplicity of notation as well as interpretability of the results, we stick to $d=2$, such that the corresponding control problem reads as 
    \begin{equation}\label{bracket_corollary}
        \begin{cases}
        \min_{a\in L^2([0,1]\R^d)}\ \ J(a) + \gamma  f(\xi(1),\Bar{g}), \\
        \dot{\xi}(s) = \sum_{i=1}^2 a_i(s) U_i(\xi(s)) + b(s) V(\xi(s)) ,\ s\in[0,1], \\
        \xi(0) = (1,0,\dots,0)\in G^N(\mathbb{R}^d),
        \end{cases}
    \end{equation}
where the $d$ vector fields $U_1,\ldots,U_d$ are defined as in \eqref{vec_fields} and 
    \begin{equation}\label{vec_field_bracket}
        V:{G}^N(\mathbb{R}^d) \rightarrow T {G}^N(\mathbb{R}^d),  V:\xi(s) \mapsto \xi(s) \otimes [e_1, e_2].
    \end{equation}
This means, we just add valid vector fields which are included in the underlying Lie algebra to the control system. The difference is the interpretation of the coefficient $b \in L^2([0,1],\R)$ in front of the vector field $V$. As before, we try to find a path $y\in H^1([0,1],\R^2)$ which can be related to the weight $b(s)$, similar to the case of the simple vector fields $U_1,\ldots,U_d$, where $a_i := [\dot{x}]_i$. 
To do that, we first consider the path whose derivative within the interval $[s,s+\Delta ]$ is given by 
\begin{equation}\label{bracket_path}
\dot{y} = \begin{cases}
    c e_1 & t\in[s,s+\Delta/4],\\
    c e_2 & t\in[s+\Delta/4,s+\Delta/2],\\
    -c e_1 & t\in[s+\Delta/2,s+\Delta3/4],\\
    -c e_2 &  t\in[s+3/4 \Delta,s+\Delta].
    \end{cases}
\end{equation}
The computations of its first- and second-order signature yields the following tensors:
    \begin{align*}\pi_1(S(y)_{s,s+\Delta }) = \begin{bmatrix} 
	0  \\
	  0  \\
		\end{bmatrix}, \qquad
    \pi_2(S(y)_{s,s+\Delta }) = \bigg(\frac{c \Delta }{4}\bigg)^2 \begin{bmatrix} 
	0 & 1 \\
	-1 & 0\\
		\end{bmatrix} = \bigg(\frac{c\Delta }{4}\bigg)^2 [e_1,e_2].
  \end{align*}
The last equality shows that the signature of a path whose derivative is of the type of \eqref{bracket_path} includes exactly the bracket expression $[e_1,e_2]$. Thus, the weight in front of the bracket can be related to a path in $y\in H^1([0,1],\R^2)$ as desired. Consequently, when we have found the optimal weight $b(s)$ for a $ s\in [0,1]$ in \eqref{bracket_corollary} we can derive the corresponding coefficients $c\in \R$ in the path \eqref{bracket_path} as follows. Again we assume that we solve the control problem on a discrete grid $\{0,1/D,\dots,1\}$ such that the optimal control $b(t)>0$ is fixed over a period $t\in [s,s+\frac{1}{D}]$, where $s\in \{0,1/D,\dots,1\}$. The correct weight $c:=\tilde{c}(s)$ for the path over the same period given via \eqref{bracket_path} can be derived as
    \begin{equation} \label{opt_b}
        b(t)1/D = \bigg(\frac{\tilde{c}(t)1/D}{4}\bigg)^2  \iff \tilde{c}(t) = 4\sqrt{\frac{b(t)}{1/D}}\ ,\ \forall t\in[s,s+1/D]
    \end{equation}
where $\Delta = 1/D$. The last identity relates the magnitude of the control $b(t)$ with the length of the path $y$ that induces on $G^N(\R^d)$ a movement along the direction $b(t)V(\xi)$. 
Finally, we observe that, if $b(t)<0$, it is sufficient to parametrize \eqref{bracket_path} in the opposite direction.\\
The procedure for optimizing the corresponding weight follows analogous steps as in the case without the final variable time. We list differences and similarities below. 
\begin{itemize}
    \item First of all we have to include into the energy functional $J$ the cost related to the control $b$. 
    Keeping in mind the relation \eqref{opt_b}, propose the following energy cost:
        \begin{equation*}
        J_{[\cdot]}(a,b) = \frac{1}{2} \int_0^1 |a(t)|^2 + 16D|b(t)| \,dt.  
        \end{equation*} 
    \item The optimization for the $p$ trajectory is not affected by the adaption of the energy functional, but $\frac{\partial}{\partial \xi_{i_1,\dots,i_k}} H(\xi,(a,b),p)$ changes. Therefore, we repeat part a) for the case $k\in \{2,\dots,N-1\}$ and derive the following derivative for the version of this section. Again we use the multiindex $j$ such that $|j|=k$ and $j_i$ is the $i$-th component of $j$. We stick to dimension $d=2$ for readability such that $b\in \R$ is the single coefficient in front of the vector field $V$ in \eqref{vec_field_bracket}:
    \begin{align*}
        & \frac{\partial}{\partial \xi_{j}} H(\xi,(a,b),p) \\
        &\stackrel{\eqref{Hamiltonian}-\eqref{time_derivative}}{=} \frac{\partial}{\partial \xi_{j}} \bigg(\gamma \frac{\sum_{k=0}^N \sum_{|l|=k}  (\xi_l-\Bar{g}_l)}{\sqrt{|\xi(s)-\Bar{g}|^2+1}} ( a_{j_k} \xi_{l \setminus j_k} + b\  \xi_{l \setminus \{j_k,j_{k-1}\}} \mathbbm{1}_{j_{k-1}\neq j_k} (-1) \mathbbm{1}_{j_{k-1} > j_k} \\
        &\qquad+  \langle p, \sum_{i=1}^2 a_i \xi \otimes e_i  +b (\xi \otimes e_1 \otimes e_2  - \xi \otimes e_2 \otimes e_1 ) \rangle \bigg)\\
        &=  \gamma \bigg[\bigg(a_{j_k} \xi_{j\setminus j_k} + b\  \xi_{j \setminus \{j_k,j_{k-1}\}} \mathbbm{1}_{j_{k-1}\neq j_k} (-1) \mathbbm{1}_{j_{k-1} > j_k} + \sum_{i_{k+1}=1}^2 (\xi_{ji_{k+1}}-  \Bar{g}_{ji_{k+1}}) a_{i_{k+1}} \\
        &\qquad+  \sum_{\substack{i_{k+1}, i_{k+1} =1 \\i_{k+1}\neq i_{k+1}}}^2 (\xi_{ji_{k+1}i_{k+2}}-  \Bar{g}_{ji_{k+1}i_{k+2}}) b \ (-1) \mathbbm{1}_{j_{k-1} > j_k} \bigg) \frac{1}{\sqrt{|\xi-\bar{g}|^2+1}} \\ 
        &\qquad + \sum_{k=0}^N \sum_{|l|=k}  (\xi_l-\Bar{g}_l) [ a_{l_k} \xi_{l \setminus l_k} + b\  \xi_{l \setminus \{l_k,l_{k-1}\}} \mathbbm{1}_{l_{k-1}\neq l_k} (-1) \mathbbm{1}_{l_{k-1} > l_k} ] \frac{-(\xi_j-\Bar{g}_j)}{\big(|\xi-\bar{g}|^2+1\big)^{\frac{3}{2}}} \bigg] \\
        &\qquad+ \sum_{i_{k+1}=1}^2   p_{ji_{k+1}} a_{i_{k+1}} + \sum_{\substack{i_{k+1}, i_{k+1} =1 \\i_{k+1}\neq i_{k+1}}}^2 b \ p_{ji_{k+1}i_{k+2}}  (-1) \mathbbm{1}_{i_{k+1}>i_{k+2}} 
    \end{align*}
    
    \item The minimization of the controls $a_i,\ i\in\{0,\dots,d\}$ is not affected. On the other hand, the optimization for the optimal $b$ cannot be done via the derivative, since the Hamiltonian is not smooth in $b$. However, the analogous augmented Hamiltonian for $b(s)$ given in \eqref{eq:augm_ham} consists of the sum of a quadratic function in $b(s)$ with $16D|b(s)|$. Therefore, the minimization can be carried out explicitly.
    Note that we expect to observe $b(s)=0$ for many times $s\in\{0,1/D,\dots,1\}$, since the absolute value induces sparsity, and it is multiplied by a large factor when $D\gg1$.
\end{itemize}
Finally, we have to construct the right weights ${c}(s)$ of the path given via \eqref{bracket_path} from their $b(s)$ for all $s\in \{0,1/D,\dots,1\}$ by exploiting the relation \eqref{opt_b}. For the numerical investigation, we reference to the last section.

\section{Numerical experiments}

This final section is devoted to demonstrating the capabilities of the optimal control approach to recover the shortest path from a given signature. The first part describes experiments related to section \ref{find_path}, whereas the second section tries to illustrate the improvement of the second approach. Moreover, we will use the length of the path recovered in the first section as the lower bound for the final time in the second section. Both sections include experiments that follow two steps
    \begin{enumerate}
        \item Generate a path via a simulation of a stochastic process and derive its signature.
        \item Compute the shortest path with the same signature signature as 1., and compare it to the original path.
    \end{enumerate}
In general, we deal with realizations of stochastic processes that have less regularity as paths from $BV([0,1],\mathbb{R}^d)$. However, given that we only deal with discretizations and linear interpolations between the simulation points, we do not hurt the regularity of $BV$. As the underlying process, we take the Ornstein-Uhlenbeck process, theoretically defined as:

\begin{definition}[Ornstein Uhlenbeck process]
The (OU) follows the succeeding dynamic:   
\begin{equation*} dX_t = \kappa(\theta - X_t)dt + \sigma dW_t,\end{equation*} 
where $W_t$ is a standard Brownian Motion, $\sigma,\kappa >0$ and $\theta \in \mathbb{R}$.
\end{definition}
The meaning of the parameters are as follows:
\begin{itemize}
    \item $\theta$ is the long-term mean
    \item $\kappa$ is the speed of mean reversion
    \item $\sigma$ is the standard deviation of the process
\end{itemize} 
At any given $t$ where the process value $X_t$ is below its long-term mean $X_t<\theta$ the drift is positive. The value of the drift is proportional to the product of $\kappa$ and the difference between $\theta$ and $X_t$. In summary, the drift forces the process to approach the long-term mean $\theta$.  The higher the mean reversion $\kappa$, the faster the process is forced to approach its $\theta$. 

We choose a discretization of 100 steps in the examples below, the starting value of the simulated processes are $X_0=0$ through all of our experiments. For the visualizations below we chose a parameter set of $\theta = 5, \sigma=1,\kappa = 0.5$. For both approaches for finding the shortest path, we used the same set of hyperparameters, initializations of the controls, and initial trajectory of signatures. For the experiments, we used up to 4-dimensional paths and up to 5 levels of the signature. Below we visualize our experiments for the reconstruction of up to 3-dimensional paths.

\subsection{Minimizing the energy functional}\label{sec:orig_minim_numerics}

The procedure of minimizing the energy is described precisely in section \ref{find_path}. In each iteration of the algorithm, we use a grid of $D=100$ discretization points. Moreover, due to the fact, that we integrated the hard final time constraint as a soft constraint in the problem as formulated in \eqref{gamma_opt_prob}, there is a tradeoff in the objective function while minimizing. On the one hand, minimizing $J(a)$ corresponds to the problem of finding a geodesic. On the other hand, the minimization of the second part of the objective can be translated into finding a signature that coincides with the one signature $\Bar{g}$ we wanted to approximate originally. However, the numerical experiments show that there is still a gap between the approximative signature $\xi(1)$ and $\Bar{g}$. Thus we assume that the optimization routine we propose finds a curve $s\mapsto \xi(s) \in \mathbb{X}$ that is close to a geodesic, which does not match the final target exactly. In order to fill this gap the recovered curve needs more length to reach the target point $\Bar{g}.$ Thus, we calculate the value of the length from the optimal curve $s\mapsto \xi(s) \in \mathbb{X}$ and in a second step we take this value as a starting point for the approach we suggested in Section~\ref{sec:fin_var_time} where we have to choose the length in the beginning on our own.

\subsection{Solving the final variable time problem}

In this section, we confirm our assumption and improve the numerical error when we follow the approach proposed in Section~\ref{sec:fin_var_time}. More precisely, when we switch to the variable final time problem we have to guess an initial final time. Due to the reparametrization described in \eqref{reparam_obj} the final time is related to the length of the path. Thus, we take the length of the path we recovered in the section before as a lower bound for our initial guess of the final time. As these experiments show we could always improve the error of the section before. This leads us to the claim that assuming the minimizer of the section before has too little length was right.

First, we simulate a two-dimensional trajectory of the process described above. Next, we compute the signature from it and derive the shortest path for this specific signature. We visualize this experiment by plotting the original trajectory and the one corresponding to the shortest path next to each other. 

\begin{figure}[htbp]
    \begin{minipage}[b]{0.5\textwidth}
        \centering
        \includegraphics[width=\textwidth]{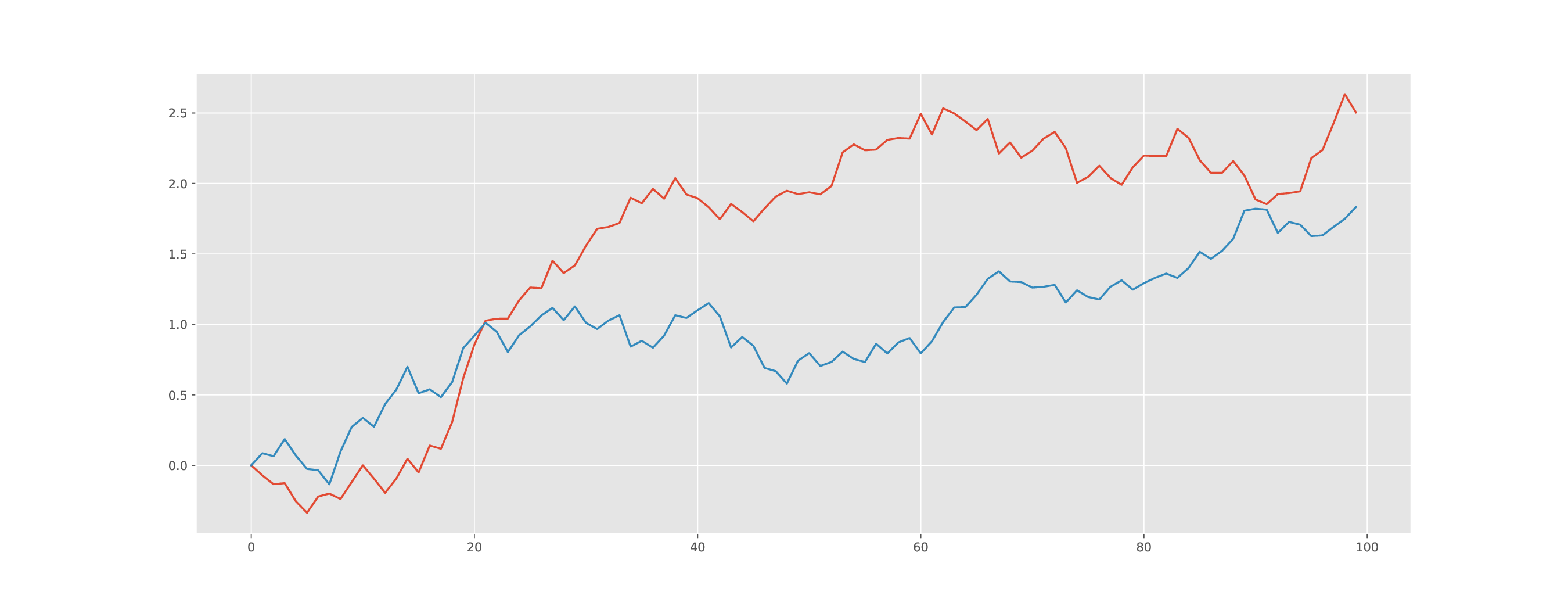}
    \end{minipage}
    \hfill
    \begin{minipage}[b]{0.5\textwidth}
        \centering
        \includegraphics[width=\textwidth]{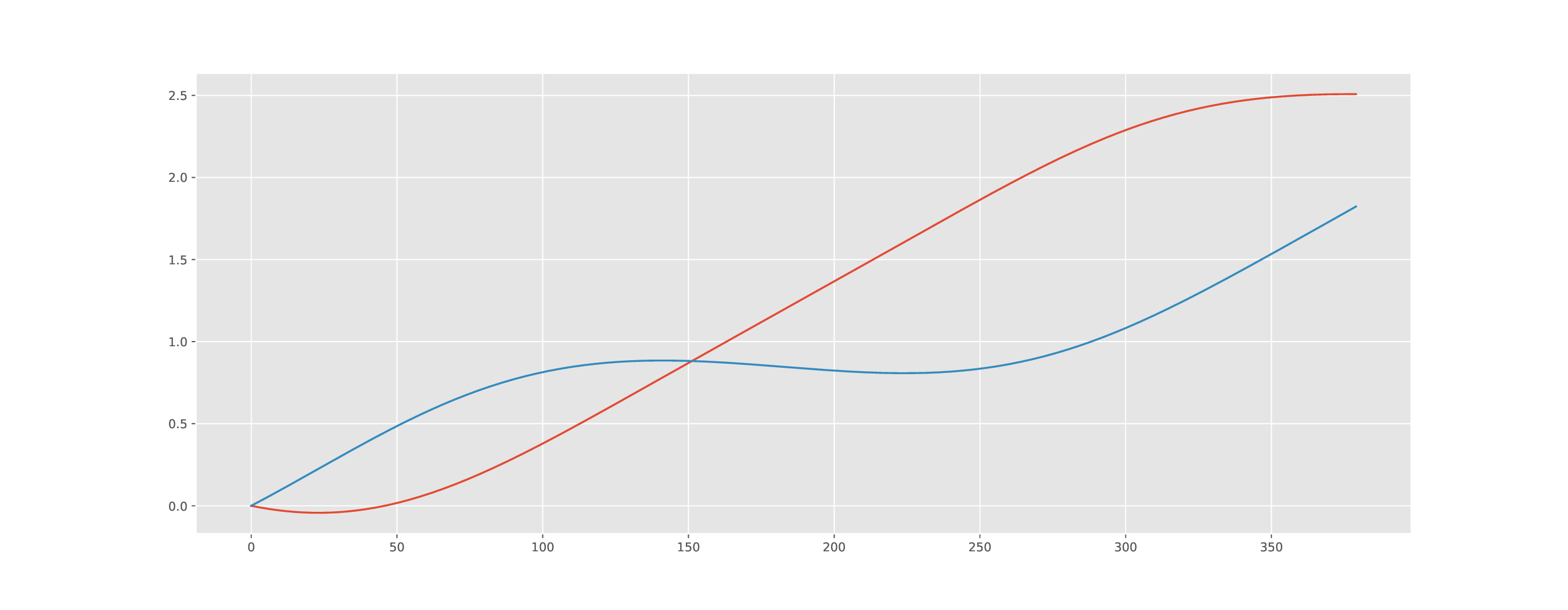}
    \end{minipage}
    \caption{Visualization of both trajectories of a 2D simulation of the OU process and the related shortest path}
        \label{fig:figure1}
\end{figure}

\begin{figure}[htbp]
    \begin{minipage}[b]{0.5\textwidth}
        \centering
        \includegraphics[width=\textwidth]{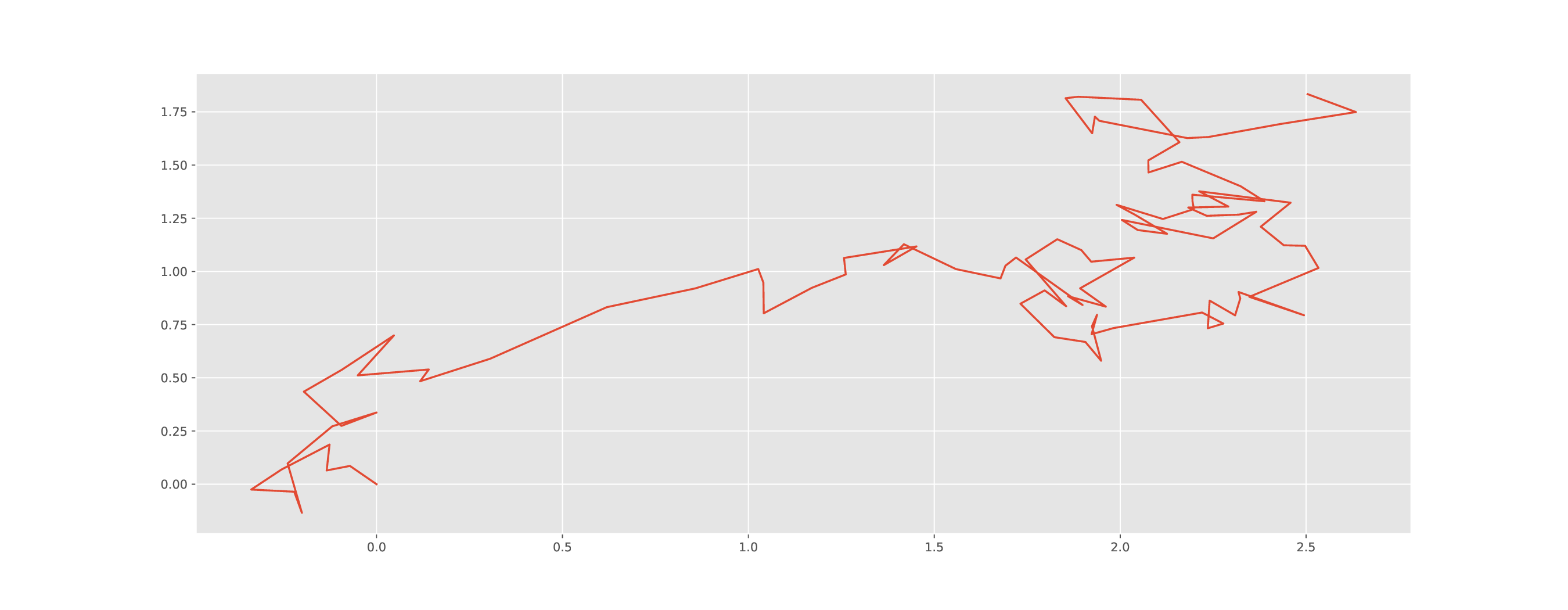}
    \end{minipage}
    \hfill
    \begin{minipage}[b]{0.5\textwidth}
        \centering
        \includegraphics[width=\textwidth]{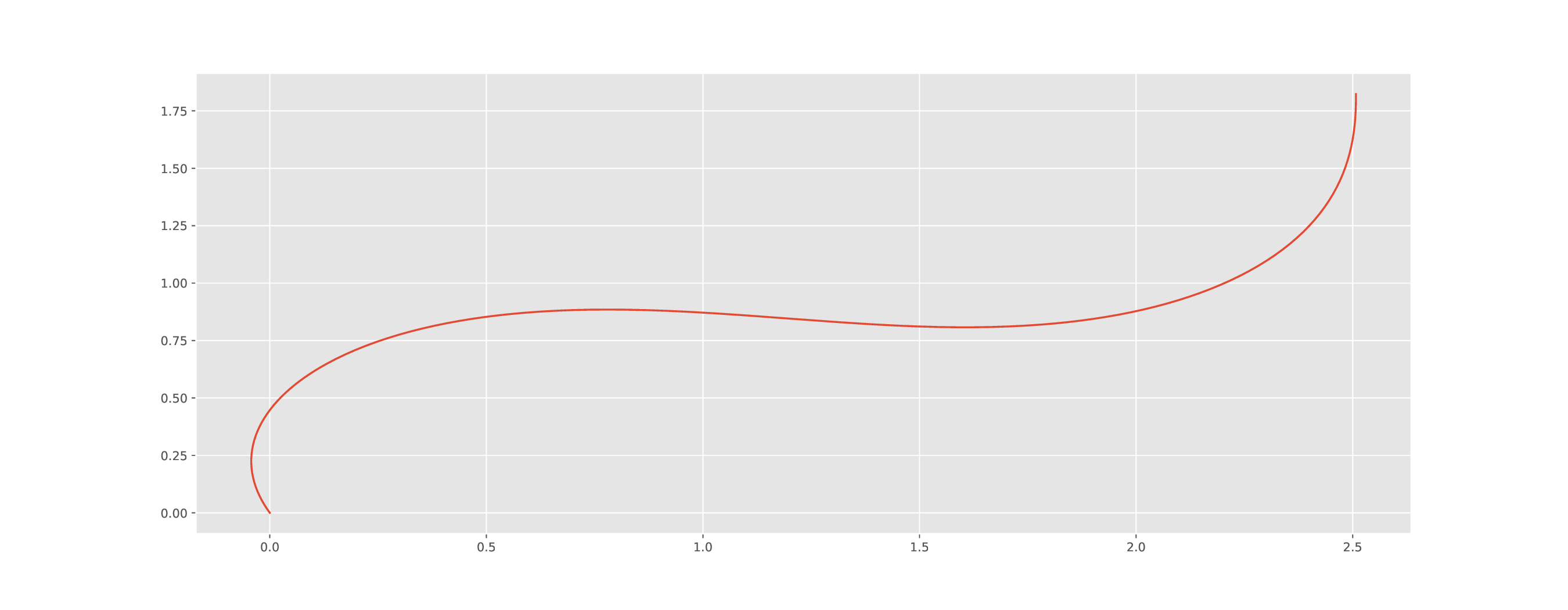}
    \end{minipage}
    \caption{Visualization of both trajectories plotted against each other of a 2D simulation of the OU process and the related shortest path}
        \label{fig:figure2}
\end{figure}

Finally, we repeat the experiment with a three-dimensional version of this process.

\begin{figure}[htbp]
    \begin{minipage}[b]{0.5\textwidth}
        \centering
        \includegraphics[width=\textwidth]{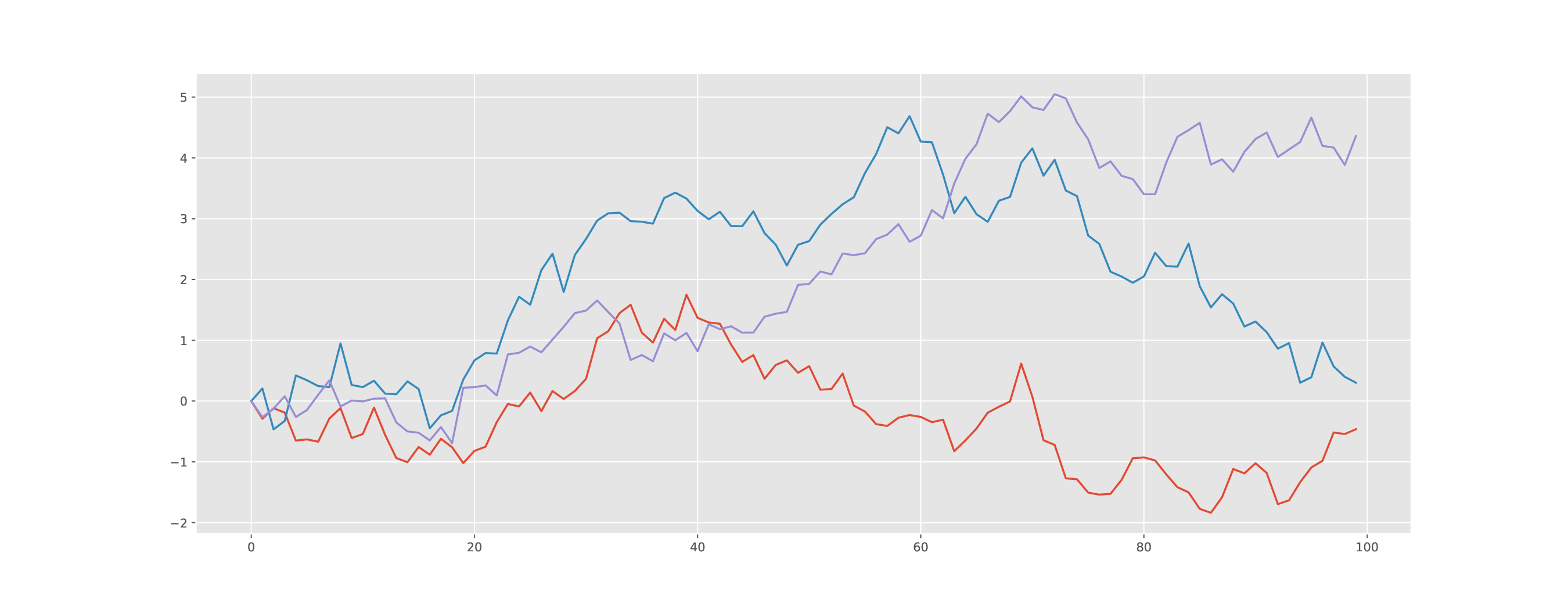}
    \end{minipage}
    \hfill
    \begin{minipage}[b]{0.5\textwidth}
        \centering
        \includegraphics[width=\textwidth]{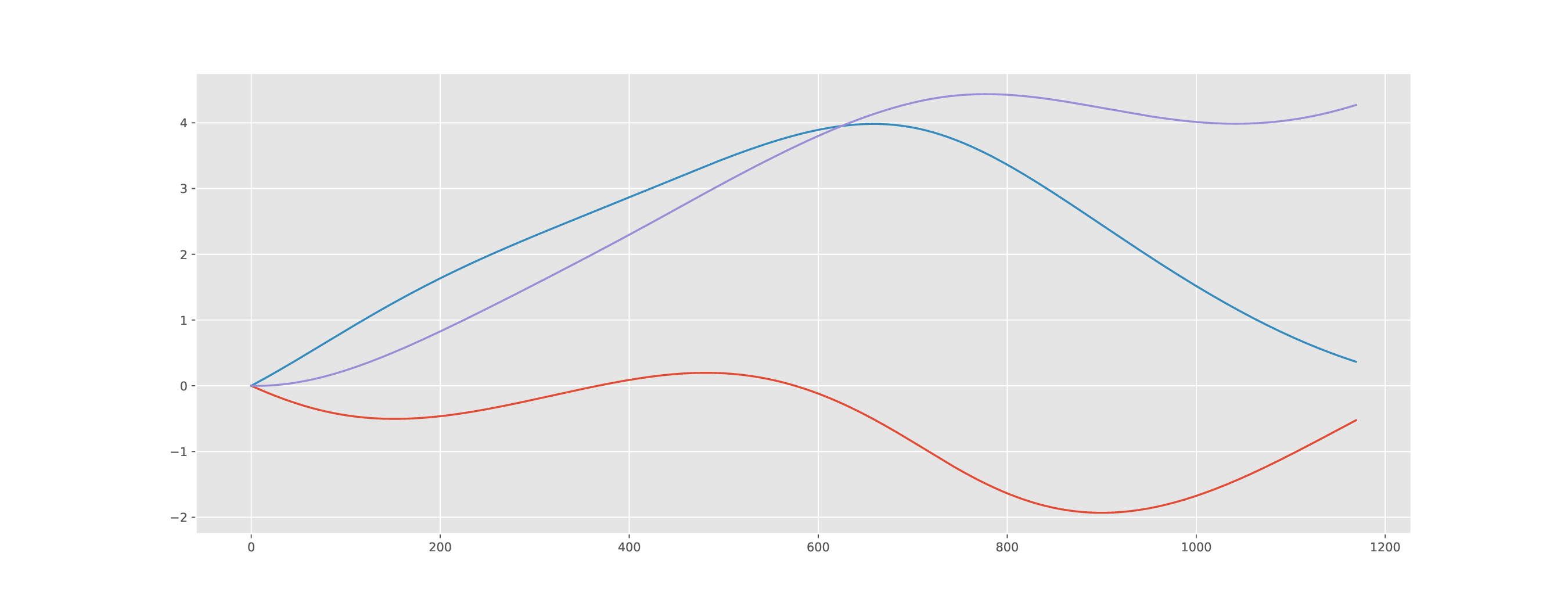}
    \end{minipage}
        \caption{Visualization of trajectories of a 3D simulation of the OU process (left) and the related shortest path (right). The graphs show the values of the components ($y$-axis) versus time ($x$-axis).}
\end{figure}
The last 3-dimensional Figure \ref{3_d_plot} plots the trajectory of the shortest path and the original path in one single figure. 

\begin{figure}\label{3_d_plot}
\begin{minipage}[b]{1\textwidth}
        \makebox[\textwidth][c]{
        \includegraphics[width=1.5\textwidth]{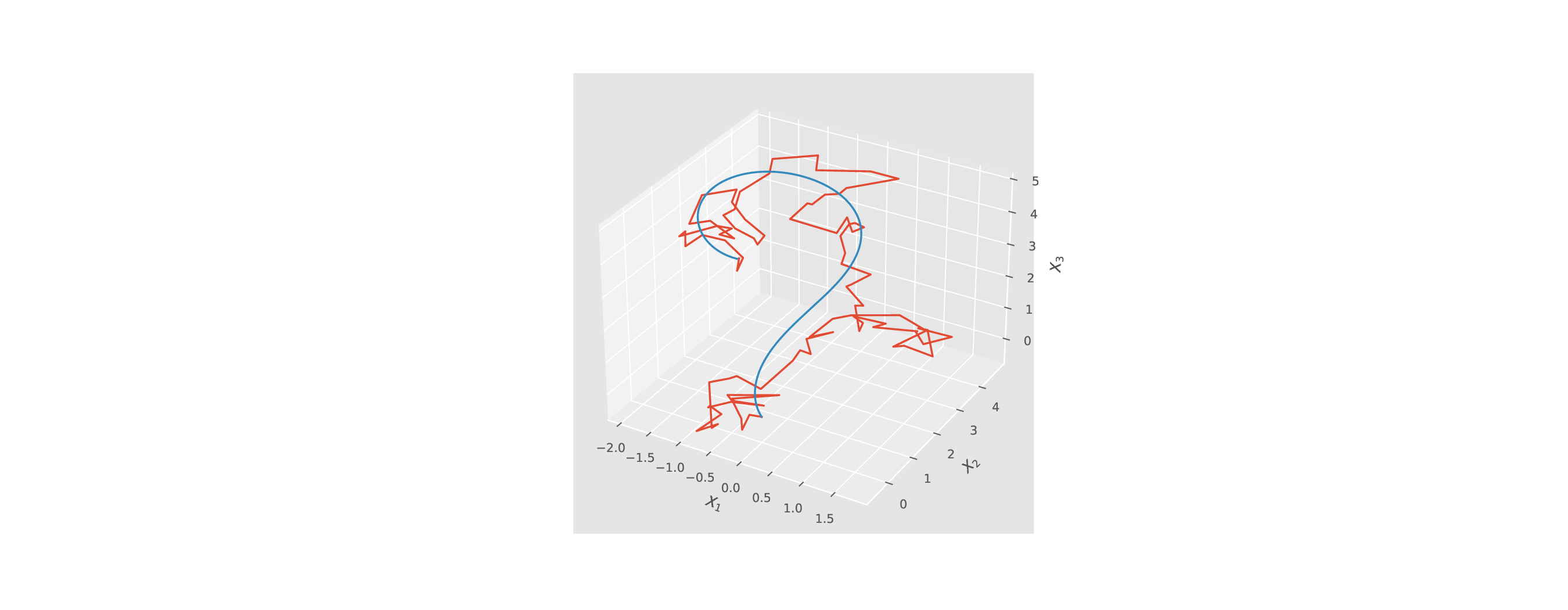}}
        \caption{Visualization of all trajectories of a 3D simulation of the OU process and the related shortest path, where we plot all trajectories against each other.}
\end{minipage}
\end{figure}

\subsection{Minimizing the Version with with Additional Vector Fields}
When considering an augmented control system that includes the bracket-type vector fields as described in \eqref{bracket_corollary}, 
the result does only very rarely change compared to Section~\ref{sec:orig_minim_numerics}. Thus, we can confirm that the weight in front of the additional bracket vector field in the control system is pushed heavily to $0$. In other words, optimal solutions prefer a sparse activation of the additional vector fields, as we suggested in Section~\ref{sec:add_vec}. For this reason, we observe a very similar result as before.

\section*{Conclusion}

In this paper, we are interested in recovering the shortest path having a prescribed signature, and we highlight the intersection between the field of optimal control and signatures. 
We address this task by formulating different optimal control problems, and we propose implementable iterative schemes based on the Pontryagin Maximum Principle.
Even though the signaure-to-path recovery has already been treated in the literature, we contributed to the field by taking advantage of the viewpoint of sub-Riemannian geometry, and by rephrasing the reconstruction as a geodesic problem.

\bmhead{Acknowledgments}

A.S. acknowledges partial support from INdAM--GNAMPA.
We want to thank two anonymous Referees whose comments helped us improve the quality of the exposition.



\appendix

\section{Proof of Theorem~\ref{gamma_conv_thm}}

As we mentioned in the paper we prove the $\Gamma$-convergence result in a slightly more general setting than in \cite{scagliotti2023gradient}, namely when the controls are allowed to be taken from the space $L^p([0,1],\R^d),\ p\in(1,\infty)$. As a matter of fact, the underlying space of curves in $G^N(\R^d)$ can be generalized to Sobolev spaces of the the form $\mathbb{X}^{1,p} = W^{1,p}([0,1],G^N(\R^d))$ with $ p\in(1,\infty)$.

\begin{proof}[Proof of Theorem~\ref{gamma_conv_thm}]
The proof can be divided into three steps.

Step~1. First of all, it is proved that a weakly convergent sequence of the controls leads to a strongly $C^0$-convergent sequence of the states, i.e. $a_m\weak_{L^p} a$ implies that $\xi_m \rightarrow_{C^0} \xi$ as $m\to\infty$, where $\xi$ is the signature solving the corresponding control system defined in \eqref{eq:control_syst} driven by the limiting control $a$. 
To prove this, we consider a sequence of states $(\xi_m)_{m>0}$ that solves the control system \eqref{eq:control_syst} along the weakly convergent sequence of controls $(a_m)_{m>0}$. 
Note that the bound in \eqref{bound_g} also holds when we consider the $L^p$-norm in place of $L^2$. Combining this observation with \eqref{assumpt_bounded_sig}, we deduce that the sequence of states contains a weakly converging subsequence $\xi_{m_n} \rightharpoonup_{W^{1,p}} \tilde{\xi}$, i.e. it is relatively compact w.r.t. to the weak topology of $\mathbb{X}^{1,p}$. 
Moreover, such a subsequence converges strongly in $C^0$ due to the compact embedding $\mathbb{X}^{1,p}\hookrightarrow C^0$. We have to show that $\tilde{\xi}$ coincides with the solution $\xi$ of the control system driven by the limit of the controls sequence $(a_m)_{m>0}$.
Since the vector fields $U_1,\ldots,U_d$ are Lipschitz-continuous, we obtain the strong convergence of the sequence $\big(U_i(\xi_{m_n})\big)_n$ w.r.t. the topology of $C^0$. 
From this fact, it descends the following weak convergence:
\begin{equation*}
    \dot \xi_{m_n} = \sum_{i=1}^d [a_{m_n}]_i U_i(\xi_{m_n}) \weak_{L^p} 
    \sum_{i=1}^d [a]_i U_i(\xi) =
    \dot{\tilde \xi}.
\end{equation*}
Recalling that $\xi(0)=\tilde \xi(0)=(1,0,\ldots,0)\in G^N(\R^d)$, we conclude that $\tilde{\xi}$ is equal to the solution $\xi$ corresponding to the control $a$. 
This argument shows that any limiting cluster point (in $W^{1,p}$ and in $C^0$) of $(\xi_m)_m$ actually coincides with $\xi$, i.e., the whole sequence $(\xi_m)_m$ converges, without needing to extract a subsequence.

Step~2. We first introduce the set
    \begin{equation*}
        A_c=\{a\in L^p([0,1],\R^d), ||a||_{L^p}\leq c\}, \ c>0
    \end{equation*}
and we show that the objective functional $\mathcal{F}_\gamma$ restricted to such a domain always admits a minimizer $a^*\in A_c$. In order to exploit the Direct Method of Calculus of Variations we need to show that $\mathcal{F}_\gamma\big|_{A_c}$ is lower semi-continuous as well as coercive. Owing to the continuity of the regularizing function and the first step we have on the one hand the convergence of $f(\xi_m(1),\bar{g})\rightarrow f(\xi(1),\bar{g})$. On the other hand, considering again the weakly convergent sequence of controls we have $\|a\|_{L^p}\leq \liminf_{m\to \infty} \|a_m\|_{L^p}$. Combining these two facts gives the lower semi-continuity. Additionally, the coercivity follows from the compactness of the domain $A_c$, such that we can deduce the existence of a minimizer of the restricted problem by the Direct Method of Calculus of Variations. 

Step~3. We are finally in position to prove that $\mathcal{F}_\gamma$ $\Gamma$-converges on the restricted domain $A_c$ to the functional $\mathcal{F}$ defined in \eqref{eq:gamma_limit}. 
Since $\mathcal{F}(a)$ does not contain the terminal cost, we can deduce that the signature corresponding to the limiting curve $a\in L^p([0,1],\mathbb{R}^d)$ has to fulfill $f(\xi(1),\Bar{g})=0$, i.e. $\xi(1)=\Bar{g}$. 
To prove the $\limsup$-condition, we set the recovery sequence constant to $a_\gamma \equiv a$ such that $\mathcal{F}_\gamma(a_\gamma)=\mathcal{F}(a)$ for all $\gamma>0$, so that the $\limsup$-inequality is naturally fulfilled. 
On the other hand, in order to establish the $\liminf$-condition, we consider w.l.o.g. a weakly convergent sequence $a_\gamma \rightharpoonup_{L^p} a$ such that $\lim_{\gamma\rightarrow \infty}\mathcal{F}_\gamma(a_\gamma)= \liminf_{\gamma\rightarrow \infty} \mathcal{F}_\gamma(a_\gamma) =C.$  According to the first step the convergence of the corresponding signatures is of the type $\xi_{\gamma} \to_{C^0} \xi$. Due to $\gamma f(\xi_\gamma(1),g) \leq \mathcal{F}_\gamma(a_\gamma)\leq C+1$ and $\gamma \to \infty $ we have $f(\xi_\gamma(1),g)\to 0$ which implies $\xi_{\gamma}(1) \to \xi(1) = \bar{g}$. Therefore, we obtain
    \begin{equation*}
        \mathcal{F}(a) = \|a\|_{L^p}^p \leq \|a_\gamma\|_{L^p}^p  \leq \liminf_{\gamma \rightarrow \infty} \mathcal{F}_\gamma(a_\gamma),
    \end{equation*}
which establishes the $\liminf$ inequality.
\end{proof}

\end{document}